%% file: domination.tex
\documentclass[12pt]{article}

\usepackage[utf8]{inputenc}
\usepackage{amsmath}
\usepackage{color}
\usepackage{xypic}
\usepackage{graphicx}
\usepackage{float}
\usepackage{ifpdf}
\usepackage{url}
\usepackage{enumerate}
\usepackage{xspace}
\usepackage{multirow}
\usepackage{hhline}
\usepackage{microtype}
\usepackage{lmodern}
\usepackage{tabularx}
\usepackage{tikz}
\usepackage{tkz-euclide}
\usetikzlibrary{calc, positioning, fit, arrows.meta, decorations.markings} 
\usepackage[above]{placeins}

\usepackage{amsmath, amsfonts, amsthm, amssymb, graphicx, hyperref, tikz} 
\usepackage[numbers,square, comma, sort&compress]{natbib}
\usepackage[english]{babel}

\usepackage[a4paper, margin=2.5cm]{geometry}

\newcommand{\dom}[1]{\gamma\left(#1\right)}
\newcommand{\idom}[1]{\gamma_i\left(#1\right)}
\newcommand{\domfrac}[1]{\frac{\gamma_i\left(#1\right)}{\gamma\left(#1\right)}}
\newcommand{\domfracC}[1]{\frac{\gamma_i}{\gamma}({#1})}
\newcommand{\domfracCinf}[1]{\frac{\gamma_i}{\gamma}({#1})_{\infty}}

\usepackage{booktabs} 
\usepackage{makecell} 

\newtheorem{theorem}{Theorem}
\newtheorem{lemma}[theorem]{Lemma}
\newtheorem{remark}[theorem]{Remark}

\newtheorem{corollary}[theorem]{Corollary}

\newtheorem{definition}{Definition}

\numberwithin{equation}{section}

\author{Gunnar Brinkmann, Steven Van Overberghe\\[4pt]
\normalsize Department of Mathematics, Computer Science, and Statistics, Ghent University, Belgium}

\title{Bounds for the ratio between the \\ independent domination number and the \\ domination number}

\date{}

\newcommand{\keywords}[1]{\par\smallskip\noindent\textbf{Keywords:} #1}
\begin{document}

\maketitle
\begin{abstract}
  In this article we present new and improved results for the ratio between the independent domination number and the domination number in graphs with bounded degree.
  We present a general formula that, for a fixed maximum degree, allows one to compute an upper bound for this ratio as a function of an upper bound $\beta |V|$ for the independent domination number. We also apply this formula to several known upper bounds for the independent domination number.
  Furthermore, we present constructions giving lower bounds for the best possible upper bound in various classes of graphs with bounded degree.
\end{abstract}



\section{Introduction}

  Unless explicitly mentioned otherwise, all graphs in this article are nonempty, simple, undirected graphs.  Let $G=(V,E)$ be a graph and
  $V'\subseteq V$. We denote the subgraph induced by $V'$ as $G[V']$, the set of neighbours of a vertex $v\in V$ as $N(v)$, and define $N[v]=N(v)\cup \{v\}$.
  We write $N(V')=\bigcup_{v\in V'}N(v)$ and analogously for $N[V']$. For a graph $G$ we write $V(G)$ for its set of vertices.

  A dominating set $D$ in a graph $G=(V,E)$ is a set of vertices such that each $v\in V$ is in $D$ or has a neighbour in $D$. Domination is an extremely well-studied
  topic in graph theory with many applications. Instead of giving only a tiny glimpse of the field, we refer the reader to \cite{dombook} for a
survey. From a local or greedy point of view, independent sets seem to be good candidates for small dominating sets, as neighbouring vertices dominate
each other, which is wasteful. Nevertheless, looking e.g.\ at a tree with $n$ vertices containing two neighbouring vertices with degree $\frac{n}{2}$, it is
obvious that an independent set dominating all vertices can sometimes be much larger than a dominating set, where independence is not required.

For a graph $G$, we denote the size of a smallest independent dominating set as $\idom{G}$ and  the size of a smallest dominating set as  $\dom{G}$.
With this notation the ratio $\domfrac{G}$ in our tree example can be arbitrarily large. Graphs with large $\idom{G}$ are often graphs that are simply {\em hard to dominate} --
that is: graphs that also have a large value of $\dom{G}$. Nevertheless, the tree example shows that there are also graphs that are easy to dominate,
but not with an independent set. So we are interested in bounds on the ratio of $\domfrac{G}$ for graphs in certain classes of graphs. Or, more precisely:

\begin{definition}
  For a nonempty class $\cal C$ of graphs we define (in $\mathbb{R}\cup \{\infty\}$)

  $\domfracC{\cal C}= \sup \{ \domfrac{G} | G\in \cal C\}$.
    
\end{definition}

In the tree example the reason for the difference is clear: some vertices can dominate many more vertices than
others. This is not the case for regular graphs and only the case in a limited way for graphs with bounded degree. For fixed maximum degree, it is also easy
to see that there is an upper bound on $\domfrac{G}$, but determining a good upper bound is a difficult problem. Some upper bounds for this ratio have been
determined in earlier papers, e.g.\  in \cite{indepdom_no4}, \cite{indepdom_regular}, or \cite{dom-idom-cubic}. In this article we give  new results, theoretical as well as 
  computational, on $\domfracC{\cal C}$ for various classes ${\cal C}$.

  \section{Upper bounds for $\domfracC{\cal C}$ for some classes ${\cal C}$}

   \begin{definition}

     For $k \in \mathbb{N}$, ${\cal C}_k$ denotes the class of all graphs with maximum degree at most $k$. 
   
  For a class ${\cal C}$ such that ${\cal C}\subseteq {\cal C}_k$ for some $k$, $\sigma({\cal C})$ denotes the smallest real number such that for each graph $G\in {\cal C}$ and
  each dominating set $D$ of $G$ with $n_0$ vertices of degree $0$ in $G[D]$ and $n_{1}$ vertices of degree at least $1$ in $G[D]$
  (so $|D|=n_0+n_{1}$) there is an independent dominating set $D_{\text ind}$ with $|D_{\text ind}|\le |D|+\sigma({\cal C}) n_{1}$.

  We write $\sigma_k=\sigma({\cal C}_k)$.
  
   \end{definition}

The following Theorem~\ref{thm:fraction} uses and generalises techniques of proofs from \cite{dom-idom-cubic} (Theorem 3), \cite{indepdom_no4} (Theorem 1.8), and
\cite{indepdom_regular} (Theorem 1.4).

\begin{theorem}\label{thm:fraction}
	Let $G$ be a graph on $n$ vertices in a class ${\cal C}_k$ and with $\idom{G}\leq\beta n$.
	Furthermore, let $\sigma'\ge \sigma_k$ and $\sigma'>0$. If $1+\sigma' \ge k\beta$, then
	\[
	\domfrac{G}\leq \frac{\beta(\sigma' k + \sigma' + 1)}{\beta + \sigma'}
	\]
	and in any case \(\domfrac{G}\leq 1+\sigma'\).
\end{theorem}
\begin{proof}

  As a vertex dominates at most $k+1$ vertices (including itself), we have $\beta \ge \frac{1}{k+1}$ and therefore
  $ \frac{\beta(\sigma' k + \sigma' + 1)}{\beta + \sigma'} \ge 1$ with equality if and only if $\beta= \frac{1}{k+1}$, which implies $\idom{G} = \dom{G} = \frac{n}{k+1}$.
  So in this case, both sides of the inequality are $1$ and from now on we may assume $\beta> \frac{1}{k+1}$.

    Let $D$ be a minimum dominating set in the graph $G$ with $n$ vertices.
    Let $D_0$ be the set of $n_0$ vertices of degree $0$ in $G[D]$ and $D_{1}$ be the set of $n_{1}$ vertices of degree at least $1$ in $G[D]$.
    If $n_1=0$, then $D$ is independent, $\domfrac{G}=1$, and the inequality holds, so assume $n_{1}\ge 1$.

We distinguish two cases: if $n_0$ is large, the dominating set is {\em almost independent} and we use the bound we get from $\sigma'$.  If $n_0$ is small, we
use the fact that for vertices that are not isolated in $G[D]$, there are fewer vertices dominated only by them, so that in fact already $\dom{G}$ is relatively large.  We assume a value 
$\alpha$ such that $\alpha\cdot n_{1}$ is the border between these two cases. The exact value of $\alpha$ will be
determined later depending on $k, \sigma'$, and $\beta$.

We use
\begin{description}
\item[(a)] $\idom{G}\leq n_0 + n_{1} + \sigma'\cdot n_{1}$ and
\item[(b)] $\dom{G}= n_0 + n_{1} $.
\end{description}

As each vertex in $D_0$ dominates at most $k+1$ vertices (including itself) and each vertex in $D_{1}$ dominates at most $k$ vertices different from its neighbours in $D$ (but including itself), we have 
\begin{description}
  \item[(c)] $n\le  (k+1)n_0 + kn_{1}$ and therefore
\item[(d)] $n_0 \ge \frac{n-kn_{1}}{k+1}$
\end{description}

The following two cases overlap for $n_0 = \alpha\cdot n_{1}$, but this causes no problems.

\begin{description}
		\item[{\bf case 1:} $n_0 \ge \alpha\cdot n_{1}$]
\begin{align*}
  \domfrac{G} & \overset{(a),(b)}{\le} \frac{n_0+n_{1}+\sigma'\cdot n_{1}}{n_0+n_{1}}=1+\frac{\sigma'\cdot n_{1}}{n_0+n_{1}} \\
   &  \overset{\mbox{\tiny{case 1}}}{\le}  1+\frac{\sigma'\cdot n_{1}}{(1+\alpha)n_{1}} =  \frac{1+ \alpha + \sigma'}{1+\alpha} 
\end{align*}

	\item[{\bf case 2:} $n_0 \leq \alpha\cdot n_{1}$] 

  With (b) and (d) we get
\begin{description}
\item[(e)] $\dom{G}\ge \frac{n-kn_{1}}{k+1} + n_{1} = \frac{n + n_{1}}{k+1}$
  \end{description}

and in case 2 with (c) also

\begin{description}
\item[(f)] $n\le \alpha (k+1) n_{1} + k n_{1}$ and therefore $n_{1} \ge \frac{n}{(\alpha +1) k+\alpha}$
  \end{description}

  Using the assumption of the theorem, we get

  \begin{align*}
  \domfrac{G} & \overset{(e)}{\le} \frac{(k+1) \beta n }{n+n_{1}} \overset{(f)}{\le}  \frac{(k+1) \beta n }{n + \frac{n}{(\alpha+1)k +\alpha}} \\
  &   = \frac{((\alpha +1)k + \alpha)(k+1)\beta}{(\alpha +1)k + \alpha +1} = \frac{((\alpha +1)k + \alpha)(k+1)\beta}{(\alpha +1)(k + 1)} \\
  & = \frac{((\alpha+1)k+\alpha)\beta}{\alpha+1}
\end{align*}
  
\end{description}

The bound $f_1(\alpha)=\frac{1+ \alpha + \sigma'}{1+\alpha}$ is decreasing with $\alpha$, while the bound
$f_2(\alpha)=\frac{((\alpha+1)k+\alpha)\beta}{\alpha+1}$ is increasing with $\alpha$, but for $\alpha=0$ and $1+\sigma' \ge k\beta$ -- as required by the theorem -- we have
$f_1(0) \ge  f_2(0)$, so the best bound is obtained by choosing $\alpha$ such that $f_1(\alpha) = f_2(\alpha)$, which is the case for

	\[
	\alpha = \frac{\sigma' - \beta k + 1}{\beta k + \beta - 1}
	\]

As we may assume $\beta (k+1) > 1$, the denominator is positive.
        
Inserting this into $f_1()$ or $f_2()$, we get the bound from the theorem.

If the condition $1+\sigma' \ge k\beta$ is not fulfilled, $f_1(\alpha) < f_2(\alpha)$, so the best choice is to make sure that we always have case 1, so choose
$\alpha=0$. But this gives the trivial bound $\domfrac{G}\leq 1+\sigma'$, which is immediate by the definition of $\sigma'$.
\end{proof}

The following Remark~\ref{rem:trivlowersigma} shows that the requirement $1+\sigma' \ge k\beta$ in
 Theorem~\ref{thm:fraction} is always fulfilled for regular graphs, as $\beta \le \frac{1}{2}$ can be assumed.

\begin{remark}\label{rem:trivlowersigma}

  For each $k$ we have $\sigma_k\ge \frac{k-2}{2}$, so $1+\sigma_k \ge k\beta$ for all $\beta \le \frac{1}{2}$.

  This implies that for a $k$-regular graph with an independent dominating set of size $\beta n$, we have $1+\sigma_k \ge k\beta$.

\end{remark}

\begin{proof}
  For $k \le 2$, the result is trivial, so assume $k>2$.  The bound $\sigma_k\ge \frac{k-2}{2}$ can be deduced from a minimum dominating set in $K_{k,k}$.
  This gives $1+\sigma_k \ge 1+\frac{k-2}{2} = k \cdot \frac{1}{2} \ge k\beta$.

Counting outgoing edges in any independent set of a regular graph, we see that the complement must be at least as large as the independent set, so $\beta \le \frac{1}{2}$.

\end{proof}

If the graph is not regular -- e.g.\  a $k_1$-clique in a graph with maximum degree $k>k_1$ with each vertex in the clique also adjacent to $k-k_1+1$ vertices of degree $1$ -- then $\beta$ can be close to $1$ and $1+\sigma_k < k\beta$.

While the formula looks too complicated to be attractive, it will soon turn out to be very useful in its generality. 
We will now determine an upper bound for $\sigma_k$. 

\begin{theorem}\label{thm:delta}

  Let $k\ge 3$ and ${\cal C}_{k,\chi}$ be the class of graphs with maximum degree at most $k$ and chromatic number at most $\chi \ge 2$ and let $\sigma_{k,\chi}=\sigma({\cal C}_{k,\chi})$.
  Then $\sigma_{k,\chi} \le \max_{1\le d < \chi} \frac{d (k-1-d)}{d+1}$ where the maximum is taken over integers $d$.

\end{theorem}

\begin{proof}

 Let $G \in {\cal C}_{k,\chi}$ and $D$ a dominating set. As before, let $D_0$ denote the set of isolated vertices in $G[D]$ and $D_{1}$ denote the set of vertices with degree at least $1$ in $G[D]$.

 Let $w=\max _{1\le d < \chi}\frac{d (k-1-d)}{d+1}$. We have to prove that there is an independent dominating set of size at most $|D|+w|D_{1}|$ for every dominating set $D$ of every $G\in  {\cal C}_{k,\chi}$.

 We will give a constructive proof by induction on $|D_{1}|$. For $|D_{1}|=0$,
  this is immediate, so assume that $|D_{1}|>0$ and that the result has been proven for dominating sets $D'$ if for the corresponding $D'_{1}$ we have $|D'_{1}|<|D_{1}|$.

  Let $v$ be a vertex of minimum degree $d_m$ in $G[D_{1}]$ and let $\deg()$ denote the degree in $G[D_{1}]$.

  {\bf case 1:} Assume that we have $\deg(v)=d_m<\chi$:

  We construct $D'$ by removing all neighbours $v_1,\dots ,v_{d_m}$ of $v$ in $G[D_{1}]$ from $D$ and adding
  an independent dominating subset of their neighbours outside $D_{1}$ that are not dominated otherwise.  In the worst case all these neighbours have to be added. So we have:

  \begin{align*}
     |D'| & \le |D| -d_m + \sum_{i=1}^{d_m} (k-\deg(v_i))= |D| + \sum_{i=1}^{d_m} (k-1-\deg(v_i))\le |D| + \sum_{i=1}^{d_m} (k-1-d_m) \\
     & = |D| + d_m(k-1-d_m) = |D| + \frac{d_m(k-1-d_m)}{d_m+1}(d_m+1)\le  |D| + w(d_m+1)
  \end{align*}
  
  Consequently $v$ is isolated in $G[D']$, the vertices
$v_1,\dots ,v_{d_m}$ are no longer in $D'$, and the newly chosen vertices have no neighbour in
$D'$. Hence $|D'_{1}|\le |D_{1}|-(d_m+1)$.

  By induction, for $D'$ there is an independent dominating set $D_{\text ind}$ of size at most $|D'|+w|D'_{1}|\le |D'|+w(|D_{1}|-(d_m+1))$. Putting this together we get

  $|D_{\text ind}| \le  |D| + w(d_m+1) +w(|D_{1}|-(d_m+1)) =  |D| + w|D_{1}|$ proving the claim in case 1.

  \bigskip

  {\bf case 2:} Assume that we have $\deg(v)\ge \chi$, so $\deg(u)\ge \chi$ for all $u\in D_{1}$:

  As $G$ is colourable with $\chi$ colours, so is the graph $G[D_{1}]$
  induced by $D_{1}$. Assume that it is coloured with at most $\chi$ colours and let $C$ be a colour class of $D_{1}$ with a maximum number of elements. W.l.o.g.\ assume that each vertex in
  $G[D_{1}]$ is in $C$ or adjacent to a vertex in $C$ -- otherwise it can be recoloured in the same colour as vertices in $C$. So we have 
  $|C|\ge  \frac{|D_{1}|}{\chi}$. We will remove all elements of $D_{1}^- =D_{1}\setminus C \not= \emptyset$ from $D$, of which there are at most $\frac{\chi-1}{\chi}|D_{1}|$, and add an independent dominating set of
  those elements of $N[D_{1}^-]$ that are not dominated otherwise. In the worst case this might be 
  all neighbours of vertices in $D_{1}^- $ that are outside of $D_{1}$.  In case 2 each vertex in $D_{1}^-$ has at most $k-\chi$ neighbours outside of $D_{1}$. As for each vertex $u$ of $D_{1}^-$ at most $k-\chi$ vertices
  of $N[u]$ can be in this independent set, the total number can be at most $\frac{\chi-1}{\chi}|D_{1}| (k-\chi)$. 

  So we get for the resulting independent dominating set $D'$ of $G$:

   \begin{align*}
     |D'| & \le |D| - |D^-_{1}| + (k-\chi) |D_{1}^-|  < |D| + (k-\chi) |D_{1}^-|  \\
     & \le  |D| + \frac{(\chi-1)(k-1-(\chi-1))}{\chi} |D_{1}|\le |D| + w|D_{1}| \\
  \end{align*}

   This proves the theorem.

\end{proof}

\begin{corollary}\label{cor:sigma2}

  \vspace*{-4mm}
\begin{description}
\item[(a)] For each $k\ge 2$ we have $\sigma_k=\sigma_{k,k}$.
\item[(b)] For bipartite graphs with maximum degree $k\ge 2$ we have $\sigma({\cal C}_{k,2}) = \frac{k-2}{2}$.
  \end{description}
\end{corollary}

\begin{proof}

  For (a) it is sufficient to prove the result for each component. The case $k=2$ is trivial as $\sigma_2=\sigma_{2,2}=0$. For $k>2$, Brooks' theorem implies
  that each component different from $K_{k+1}$ of a graph in ${\cal C}_{k}$ is also in ${\cal C}_{k,k}$.  A component isomorphic to $K_{k+1}$ even has
  $\sigma=0$.

  Statement (b) follows from Theorem~\ref{thm:delta}, Remark~\ref{rem:trivlowersigma}, and part (a) for $k=2$.

    \end{proof}

Of course $\sigma_k$ can be approximated by interpreting $d$ as a real variable and analytically determining the maximum for that function. Nevertheless, especially for small values, the fact that the maximum is taken
only over integers has an effect that should not be neglected -- and the exact maximum can be computed very easily. We will now formulate a
corollary for some small $k$ that will help to give upper bounds for the ratio $\domfrac{G}$ for various classes of graphs.

\begin{corollary}\label{cor:domfrac}
  Let $G=(V,E)$ be a graph with maximum degree at most $k$ and $\idom{G}\le \beta |V|$.

   \begin{description}

  \item[For $3\le k\le 6$:]  $\sigma_k = \frac{k-2}{2}$ and  if $\beta \le \frac{1}{2}$: $\domfrac{G} \le \frac{\beta(k^2-k) }{2\beta+  k -2}$.

  \item[For $7\le k\le 12$:]  $\sigma_k \le \frac{2(k-3)}{3}$ and   if $\beta \le \frac{2k-3}{3k}$: $\domfrac{G} \le \frac{\beta(2k^2-4k-3) }{3\beta+  2k -6}$.

   \item[For general $k>12$:]    $\sigma_k \le (\sqrt{k}-1)^2$ and  if $\beta \le \frac{1+(\sqrt{k}-1)^2}{k}$: $\domfrac{G} \le \frac{\beta (k^2-2(k\sqrt{k}-k+\sqrt{k}-1))}{\beta+  (\sqrt{k}-1)^2}$
      
    \end{description}

  \end{corollary}

\begin{proof}

  For $3\le k\le 12$ the results follow immediately from Theorem~\ref{thm:fraction}, Theorem~\ref{thm:delta} and Corollary~\ref{cor:sigma2} by explicitly computing a value of $d$ for which the formula in Theorem~\ref{thm:delta}
  has a maximum (that is: $d=1$ for $3\le k\le 6$ and  $d=2$ for $7\le k\le 12$) and inserting it in the formula of  Theorem~\ref{thm:fraction}.
  Equality for $\sigma_3, \dots ,\sigma_6$ -- instead of just ``$\le$'' -- follows from Remark~\ref{rem:trivlowersigma}.
  For general $k$ the result follows by computing the maximum in Theorem~\ref{thm:delta} for real $d\in [1,k-1]$. The maximum is attained at $d=\sqrt{k}-1$. 

\end{proof}

As $\beta$ is often given in the form $\beta=\frac{a}{b}$ with small integers $a$ and $b$, it is useful to give the bounds for $\domfrac{G}$ for small fixed $k$ also in this notation:
For $k=3,4,5,6$ we have $\domfrac{G} \le \frac{6a}{2a+b}$, $\frac{6a}{a+b}$, $\frac{20a}{2a+3b}$, and $\frac{15a}{a+2b}$, respectively.

It is shown in \cite{indepdom_regular} (though not in
this terminology) that $\sigma_k \le k-3$ for $k\ge 4$. For $k = 4$ this agrees with our bound, and for
$k\ge 5$ our bound is smaller. Consequently, for the ﬁxed value $\beta=\frac{k-1}{2k-1}$ used in
\cite{indepdom_regular}, we reproduce their bound on $\domfrac{G}$  for regular graphs with degree $k = 4$ and improve it
for $k \ge  5$.
Nevertheless, the main intention of the theorems in this section is to
be able to also use them for values of $\beta$ that are specific for some special classes of graphs.

Corollary~\ref{cor:domfrac} can now be used to easily reproduce results on $\domfrac{G}$ for some classes of graphs, but also to obtain new results:

\begin{corollary}\label{cor:harvest}
  Let $G=(V,E)$ be a connected graph with maximum degree $k$.

   \begin{description}

  \item[(1)] If $3\le k\le 6$, $G$ is regular, and $G\not=K_{k,k}$, then  $\domfrac{G} \le \frac{(k-1)^2}{2k-3}$.

  \item[(2)] If $7\le k\le 12$, $G$ is regular, and $G\not=K_{k,k}$, then $\domfrac{G} \le \frac{(k-1)(2k^2-4k-3)}{4k^2-11k+3}$.

  \item[(3)] If $k=3$, $G$ is regular, $G\not= C_5\square K_2$, and $G$ does not contain $K_{2,3}$ as a subgraph, then $\domfrac{G} \le \frac{9}{7}$.\\
    This contains the class of 2-connected planar cubic graphs without the 5-prism $C_5\square K_2$.

   \item[(4)] If $k=3$, and $G$ is regular without 4-cycles, then $\domfrac{G} \le \frac{5}{4}$.

  \item[(5)] If $k= 4$, then $\domfrac{G} \le 2$. Here we do not need regularity, and the bound is sharp.
      
    \end{description}

  \end{corollary}

\begin{proof}

  The results follow by inserting values of $\beta$ into the formulas in Corollary~\ref{cor:domfrac}. For the $\sigma'$ used in Corollary~\ref{cor:domfrac}, the prerequisite  $1+\sigma' \ge k\beta$ is fulfilled in each case, except in case (5) where $1+\sigma'$ is used as a bound.

  (1) and (2) follow from the bound $\beta\le \frac{k-1}{2k-1}$ in \cite{indepdom_regular}.
For $k=3$ and $k=4$, these values were already proven in 
  \cite{dom-idom-cubic} and  \cite{indepdom_regular}, respectively. The other bounds improve earlier results.

  (3) follows from the bound $\beta=\frac{3}{8}$ that is proven in \cite{38nok23}.  If the result
  in \cite{38forcubicatleast12} -- that for $|V|>10$ the bound $\beta=\frac{3}{8}$ holds without excluding $K_{2,3}$ as a subgraph -- is confirmed, the bound holds in greater generality.
  Since the paper is still a preprint at the time of writing, we mention
  the result only here and not in the statement of the corollary. The fact that a 2-connected planar cubic graph $G$  does not contain $K_{2,3}$ as a subgraph follows as $G$ would contain
  a bridge if the $K_{2,3}$ had only one outgoing edge. If it has three outgoing edges, $G$ would have a $K_{3,3}$ minor if  after removing the vertices of the $K_{2,3}$ subgraph the remaining graph is connected
  and $G$ would contain a bridge otherwise.

  (4) follows by applying the bound $\beta=\frac{5}{14}$ proven in \cite{indepdom_no4}. This bound for  $\domfrac{G}$ is also given in that paper by a proof focused on this special case and mentioned here only for completeness.

  (5): The bound $\domfrac{G} \le 2$ is the bound $1+\sigma'$ from Theorem~\ref{thm:fraction}.
  The fact that the bound is sharp is proven by e.g.\ $K_{4,4}$ or the tree on 8 vertices with two neighbouring vertices of degree $4$.
  
\end{proof}

\begin{figure}
\begin{center}
\includegraphics[width=55mm]{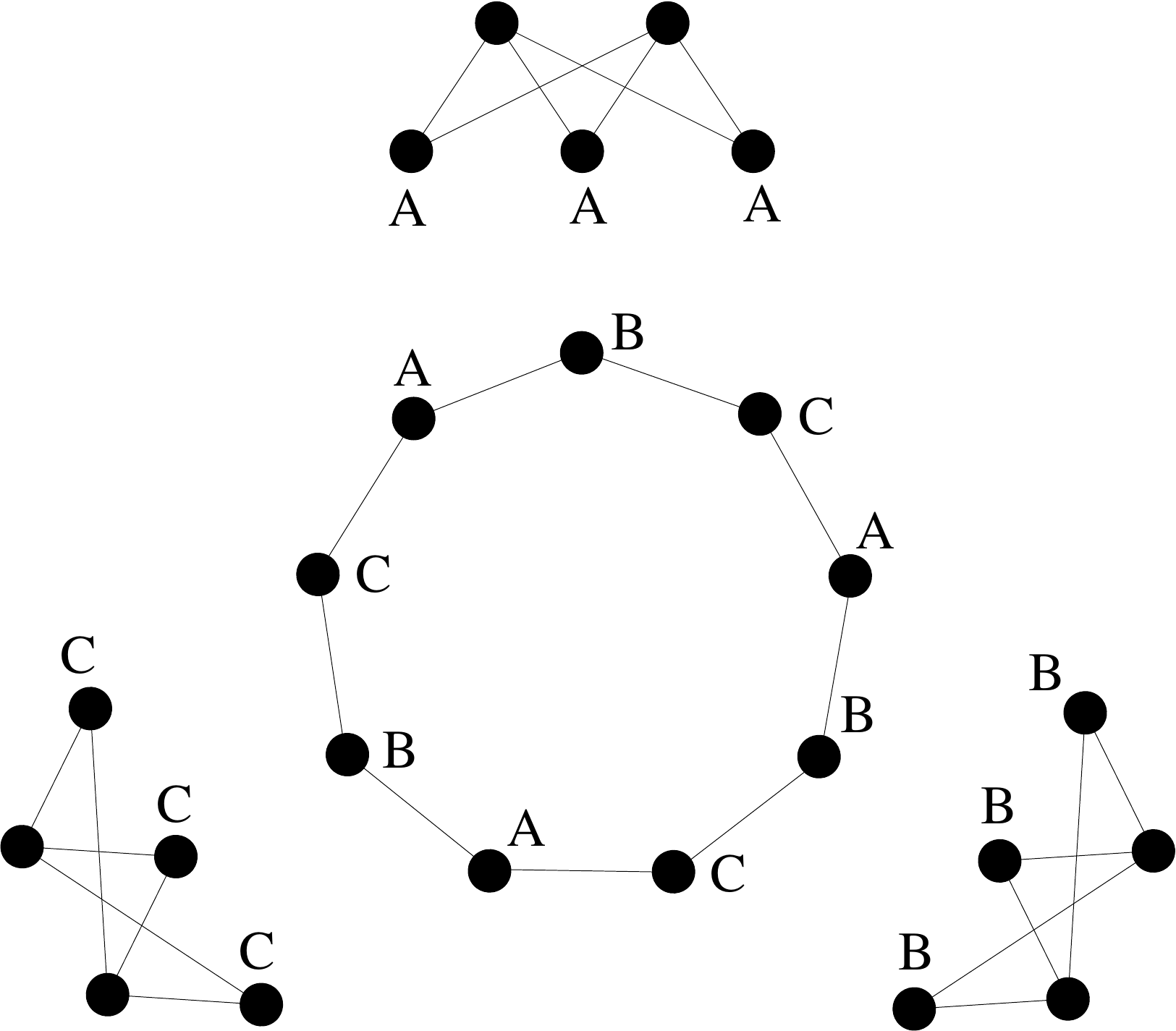}
\end{center}
\caption{A cubic, 3-connected graph $G$ on 24 vertices with $\idom{G}=9$ and $\dom{G}=7$. The labelled vertices in the copies of $K_{2,3}$ must each be connected to a different vertex in the cycle $C_9$ with the same label. This graph also shows that the answer to Question~1.1  in \cite{o_west} -- whether $\idom{G}-\dom{G}\le \frac{|V|}{16}$ for $3$-connected cubic graphs with $|V| \ge 12$ -- is negative.}\label{fig:cubic3con9_7}
\end{figure}

\section{Lower bounds for $\domfracC{\cal C}$ for some classes ${\cal C}$}

A lower bound $c$ for $\domfracC{\cal C}$ can be obtained by giving an example graph $G\in {\cal C}$ with $\domfrac{G}\ge c$. As seen in the
last section, $K_{k,k}$, and sometimes also other specific graphs like $C_5\square K_2$, are often excluded from the class, as small specific graphs can be singular
exceptions not reflecting the overall rules of a class. So we have two kinds of examples: on the one hand, small, specific graphs that leave the possibility that for
sufficiently large graphs better bounds can be proven -- and on the other hand, examples of infinite sequences $G_1,G_2,\dots $ of graphs with the same or arbitrarily close ratios
$\domfrac{G_i}$ showing that the bounds are valid independently of a finite number of exceptional graphs. We say that for a class ${\cal C}$ a sequence gives lower bounds for $\domfracCinf{\cal C}$
if it gives lower bounds for $\domfracC{{\cal C}^{\ge n}}$ for each set ${\cal C}^{\ge n}=\{G\in {\cal C} \mid |V(G)|\ge n\}$
 with $n\ge 1$, or formally

\begin{definition}
  For an infinite class $\cal C$ of graphs define

  $\domfracCinf{\cal C}= \inf \{ \domfracC{{\cal C}^{\ge n}} | n\ge 1\}$.
    
\end{definition}

Graphs presented in this section were found by structure enumeration programs like {\em snarkhunter} \cite{tricycle}, {\em genreg} \cite{MM98}, {\em minibaum} \cite{minibaum}, or {\em plantri} \cite{plantri}
and can be found in the database {\em House of Graphs} \cite{HoG2} by searching for the term {\em idomfrac}. We do not give the (sometimes lengthy) arguments showing that particular graphs have the stated domination
and independent domination numbers. Computer programs were used for this as well.
 They are ancillary files to this article, but the programs for the domination number and independent domination number can also be used for single graphs by uploading them to the {\em House of Graphs} \cite{HoG2}, which uses these programs.

\subsection*{Specific graphs}

We will start by giving some new special graphs proving lower bounds for several classes. The first graph is the cubic, 3-connected graph $G_{9,7}$ on 24 vertices with $\idom{G_{9,7}}=9$ and $\dom{G_{9,7}}=7$
    described in Figure~\ref{fig:cubic3con9_7}, so $\domfrac{G_{9,7}}=\frac{9}{7}$ and for
    ${{\cal C}_3^3}$ the class of cubic, 3-connected graphs on at least $12$ vertices, we have $\domfracC{{\cal C}_{3}^3}\ge \frac{9}{7}$.  When studying domination, the class of cubic graphs is often restricted to at
    least $12$ vertices to exclude $K_{3,3}$ and $C_5\square K_2$.  It is open whether also excluding $G_{9,7}$ would allow a smaller upper bound for the remaining class.
    $G_{9,7}$ has a very special structure, and we tried various methods to generalise this structure
    to larger graphs $G$ with an at least similarly large ratio of $\domfrac{G}$, but did not succeed.

In Figure~\ref{fig:quartic_5_3} the 4-regular, 4-connected circulant graph $C_{13}(1,3)$ on 13 vertices with  $\idom{C_{13}(1,3)}=5$ and  $\dom{C_{13}(1,3)}=3$ is given. So
$\domfrac{C_{13}(1,3)}=\frac{5}{3}$ and therefore for ${{\cal C}_4^4}$ the class of
4-regular, 4-connected graphs different from $K_{4,4}$, we have $\domfracC{{\cal C}_4^4}\ge \frac{5}{3}$. Complete enumerations show that, up to $21$ vertices, $\frac{5}{3}$ is the largest value of $\domfrac{G}$ for graphs in ${\cal C}_4^4$ and is attained only by
$C_{13}(1,3)$.

\begin{figure}
\begin{center}
\includegraphics[width=35mm]{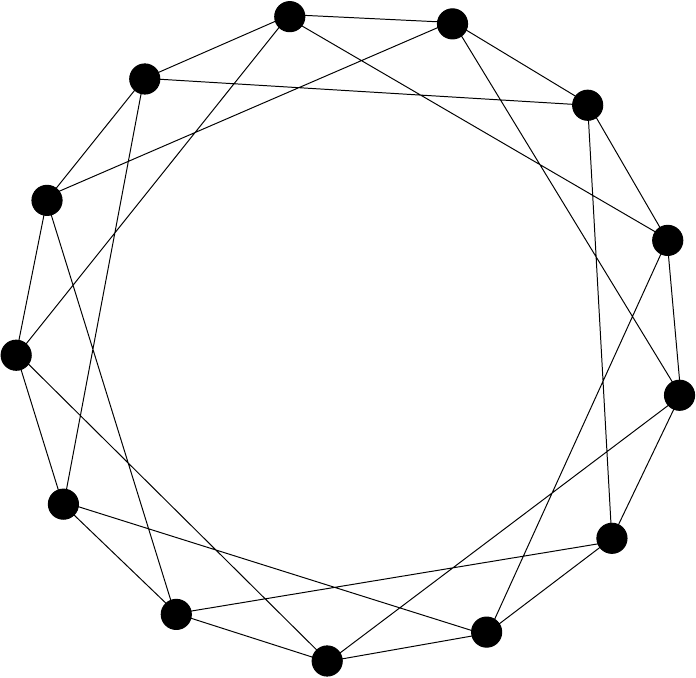}
\end{center}
\caption{A 4-regular, 4-connected graph $G$ on 13 vertices -- the circulant graph $C_{13}(1,3)$ -- with  $\idom{G}=5$ and $\dom{G}=3$. It
  has genus 1 and the embedding on the torus is unique and a quadrangulation of the torus.}\label{fig:quartic_5_3}
\end{figure}

A 4-regular, 4-connected polyhedron $G$ on 28 vertices with  $\idom{G}=8$ and  $\dom{G}=6$ is given in Figure~\ref{fig:4regpolyhedron}. So for this graph
$\domfrac{G}=\frac{4}{3}$ and therefore for ${{\cal C}_4^p}$ the class of
4-regular, 4-connected polyhedra, we have $\domfracC{{\cal C}_4^p}\ge \frac{4}{3}$. 

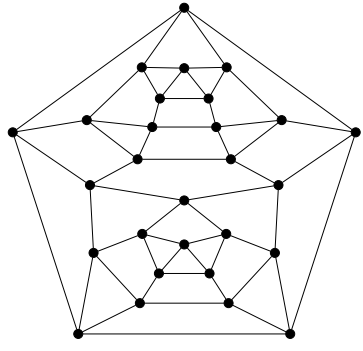
\begin{figure}
\begin{center}
      \resizebox{0.3\textwidth}{!}{
        \input{planar_4reg_4con_frac_4_3.tikz}
        }
\end{center}
\caption{A 4-regular, 4-connected polyhedron $G$ on 28 vertices with  $\idom{G}=8$ and  $\dom{G}=6$.}\label{fig:4regpolyhedron}
\end{figure}

\subsection*{Infinite sequences of graphs}

In this article, as well as in previous articles like e.g.\ \cite{indepdom_pl_cubic} or \cite{o_west}, infinite sequences of graphs are constructed by taking small building
blocks and connecting them to form arbitrarily large cycles or paths of these blocks. In some cases it is easy to see what the independent domination number of
the larger graph is (e.g.\ in \cite{o_west}), but in others (e.g.\ \cite{indepdom_pl_cubic}) it is hard to prove. As we will give several sequences based on
this cyclic construction -- some with a relatively large building block -- a case-by-case analysis by hand was not an option, and we had to develop an algorithm to
determine the independent domination number of these cycles of graphs.

First, we describe explicitly what these {\em cycles of graphs} are:

\begin{definition}

  A {\em cycle extension pair } $(G,P)$ is a graph $G=(V,E)$ together with a nonempty set $P=\{(w_1,v_1),\dots ,(w_r,v_r)\}$ of ordered pairs of vertices of $G$ with $ \{w_1,\dots ,w_r\} \cap \{v_1,\dots ,v_r\} =\emptyset$
  and $\{w_1,v_1\},\dots ,\{w_r,v_r\}\not\in E$.

  For a cycle extension pair $(G,P)$ and $p\in \mathbb{N}, p>0$, we define $(G,P)^{(p)}=(V^{(p)}),E^{(p)})$ to be the graph with $V^{(p)}=\{0,\dots,p-1\} \times V$ and

  \[
  \begin{array}{ll}
    E^{(p)}= & \{\{(x,w),(y,v)\}\ | 	\{w,v\} \in E \text{ and }x=y\} \\
    & \cup \{\{(x,w),(y,v)\}\ | (w,v)\in P \text{ and } y\equiv x+1 \pmod p\}\\
    \end{array}
  \]
      
We will denote the subgraph with vertices of the form $(i,v)$ and edges $\{(i,v),(i,w)\}$ with $\{v,w\} \in E$ as $G_i$. Indices must always be read modulo $p$.

Thus, $(G,P)^{(p)}$ consists of $p$ copies $G_0,\dots ,G_{p-1}$ of $G$ arranged in a cycle where adjacent copies are always connected in the same way. 

\end{definition}

Note that it is not required that $v_1,\dots ,v_r$ are pairwise distinct or that $w_1,\dots ,w_r$ are pairwise distinct.
As $P$ is not empty and $\{w_1,v_1\},\dots ,\{w_r,v_r\}\not\in E$, the graph $(G,P)^{(1)}$ is not isomorphic to $G$.

It is easy to see that $\dom{(G,P)^{(p)}}\le p\cdot \dom{(G,P)^{(1)}}$ and $\idom{(G,P)^{(p)}}\le p\cdot \idom{(G,P)^{(1)}}$, but we need the exact values -- especially for $\idom{(G,P)^{(p)}}$ -- to get
infinite sequences that give bounds for $\domfracCinf{\cal C}$.

\begin{figure}
\begin{center}
\includegraphics[width=75mm]{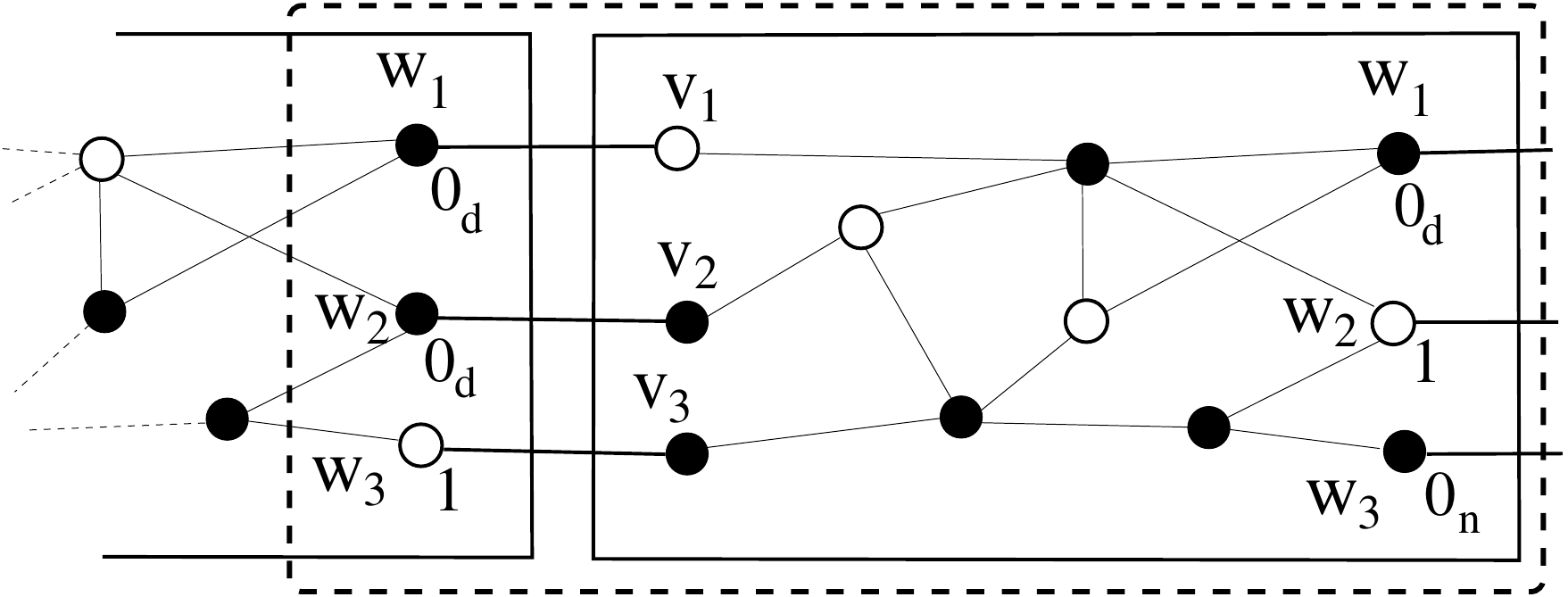}
\end{center}
\caption{A building block for a cycle of graphs. The hollow vertices denote vertices in an independent dominating set $S$. This part of $(G,P)^{(p)}$ transforms an input triple $(0_d,0_d,1)$ into an output triple $(0_d,1,0_n)$ with cost $4$. $\bar G_i$ is given by the dashed lines.}\label{fig:blok}
\end{figure}

We will describe the algorithm only for the computation of $\idom{(G,P)^{(p)}}$ -- the algorithm for $\dom{(G,P)^{(p)}}$ is completely analogous. Let a graph $(G,P)^{(p)}$ with an independent
dominating set $S$ and a copy $G_i$ of $G$ be given and let (w.l.o.g.) $w'_1=(i,w_1),\dots ,w'_q=(i,w_q)$ be the distinct vertices with $w_{j} \in \{w_1,\dots ,w_r\}$ in $P$. 

For a given independent dominating set $S$, each vertex $w\in \{w'_1,\dots ,w'_q\}$ can be characterized by exactly one of the following states $s(w)$:

\begin{enumerate}
	\item $w \in S$ (indicated as ``$s(w)=1$'')
	\item $w\notin S$, but dominated from within its own block $G_i$ (indicated as ``$s(w)=0_d$'')
	\item $w\notin S$ and not dominated from within $G_i$ (indicated as ``$s(w)=0_n$'')
\end{enumerate}

For a given building block $G_{i}$ in $(G,P)^{(p)}$, the function $s()$ defines a tuple $(s(w'_1),\dots ,s(w'_q))$ in $\{1,0_d,0_n\}^q$ -- this is the {\em
  output tuple} of $G_{i}$. The {\em input tuple} for $G_{i}$ is the output tuple of the previous block $G_{i-1}$ with indices modulo $p$. So $\bar G_i$, which is $G_i$ together with the vertices of $G_{i-1}$
corresponding to some $w\in \{w'_1,\dots ,w'_q\}$, describes a transition from the input tuple to the output tuple with a certain cost $c_i$, which is given by the number of 
vertices of $S$ in $G_i$. The vertices $v\in S$ in $\bar G_i$ but not in $G_i$ are not counted for $G_i$, as they are counted for $G_{i-1}$.
So $\idom{(G,P)^{(p)}}=\sum_{i=0}^{p-1}c_i$ and for $S$ a minimum independent dominating set, each $c_i$ is minimal for the given input and output tuple.

This process can also be reversed: having a cyclic sequence of length $p$ with compatible input and output tuples such that the sum of the transition costs is minimal, one can construct a minimum independent dominating set.

The algorithm now works as follows:

\begin{description}
\item[(a)] for all $x,y \in \{1,0_d,0_n\}^q$ compute the minimum transition cost from $x$ to $y$ -- or set it to $\infty$ if such a transition is not possible.
\end{description}

This is done by a simple branch-and-bound algorithm.

\begin{description}
\item[(b)]  Then define a weighted directed graph $H^{(G,P)}_{\gamma_i}$ (respectively $H^{(G,P)}_{\gamma}$) (possibly with loops) with vertex set
$\{(x_1,\dots ,x_q) | x_i \in \{1,0_d,0_n\}\}$. 
Note that the number of vertices of $H^{(G,P)}_{\gamma_i}$ and $H^{(G,P)}_{\gamma}$ is $3^q$ and independent of $p$ or $|V(G)|$.  The edges are the edges between vertices for which the transition
cost is finite and the weight of the edge is the transition cost. For a given $p$, $\idom{(G,P)^{(p)}}$ is given by the smallest weight of a closed directed walk of length $p$
in $H^{(G,P)}_{\gamma_i}$ -- that is, the sum of the edge weights: every independent dominating set $S$ can be translated into a walk (with the weights at most the number of
elements of $S$ in the corresponding $G_i$) and the other way around.
\end{description}

\begin{definition}
  Let $H$ be a weighted directed graph with at least one directed cycle.

  Then define \\ $W(H)= \min\{ \frac{w(C)}{l(C)} \mid C\text{ is a directed cycle in } H \text{ with length } l(C) \text{ and weight } w(C)\}$.

\end{definition}

We used an implementation of Karp's minimum mean cycle algorithm \cite{karp_minmeanweightcycle} to efficiently compute $W(H)$.

\begin{lemma}\label{lem:wH}
  Let $(G,P)$ be a cycle extension pair, $p\ge 1$, and $H^{(G,P)}_{\gamma_i}$ and $H^{(G,P)}_{\gamma}$ the corresponding weighted directed graphs.

  Then $W(H^{(G,P)}_{\gamma_i})>0$ and $W(H^{(G,P)}_{\gamma})>0$ and

  \begin{description}
  \item[(1)] $\idom{(G,P)^{(p)}}\ge p\cdot W(H^{(G,P)}_{\gamma_i})$ and  $\dom{(G,P)^{(p)}}\ge p\cdot W(H^{(G,P)}_{\gamma})$.
  \item[(2)] If $q$ is the length of a cycle realising $W(H^{(G,P)}_{\gamma_i})$, then $\idom{(G,P)^{(p)}} = p\cdot W(H^{(G,P)}_{\gamma_i})$ for all $p=q\cdot j$ with $j\ge 1$ and analogously for $q', W(H^{(G,P)}_{\gamma})$, and $\dom{(G,P)^{(p)}}$.
    \item[(3)] There are infinitely many $r$ such that $\domfrac{(G,P)^{(r)}}\ge \frac{W(H^{(G,P)}_{\gamma_i})}{W(H^{(G,P)}_{\gamma})}$ and even infinitely many $r$ such that $\domfrac{(G,P)^{(r)}}= \frac{W(H^{(G,P)}_{\gamma_i})}{W(H^{(G,P)}_{\gamma})}$.
    \end{description}
  \end{lemma}

\begin{proof}
  The statements $W(H^{(G,P)}_{\gamma_i})>0$ and $W(H^{(G,P)}_{\gamma})>0$ are obvious, as each closed walk must contain an edge with positive weight.
  
  {\bf (1):} Let $S$ be a minimum independent dominating set for $(G,P)^{(p)}$. Then $S$ defines a closed walk $C$ of length $p$ in $H^{(G,P)}_{\gamma_i}$. We prove the result by induction on $p$.
  If $C$ is a cycle, the statement follows directly from the definition. Otherwise let $C_0$ be a shortest closed subwalk -- that is: a cycle. If $C_0$ has length $m_0$, then removing the copies
  of $G$ corresponding to the edges of $C_0$ (and connecting the boundary copies) we have a graph $(G,P)^{(p-m_0)}$ with an independent dominating set corresponding to the
  remaining edges and to which we can apply induction, so

  $\idom{(G,P)^{(p-m_0)}}\ge (p-m_0)\cdot W(H^{(G,P)}_{\gamma_i})$ and as for $C_0$ we have \\ $\frac{w(C_0)}{m_0}\cdot m_0 \ge W(H^{(G,P)}_{\gamma_i})\cdot m_0$ we get the result.

  {\bf (2):} As multiples of $q$ allow independent dominating sets with minimum average cost per copy of $G$, this follows immediately. The same holds for $q'$ and dominating sets.

  {\bf (3):} The first part follows immediately with (1) and (2) for multiples of the number $q'$ in (2). The second part follows for common multiples of $q$ and $q'$.

\end{proof}

\begin{lemma}\label{lem:cycleext}
  Let $c\ge c_1 >0$ be integer numbers and $(G,P)$ a cycle extension pair with a $c_1$-connected graph $G$ and $P = \{(w_1 , v_1 ), \dots , (w_r , v_r )\}$, so that no vertex occurs in two pairs and $r \ge \lceil \frac{c}{2}\rceil $.

  If $(G,P)^{(p)}$ is $c$-connected for $1\le p \le p_{max}$ with $p_{max}=\max \{1, \lfloor \frac{c-1}{c_1}\rfloor \}$, then $(G,P)^{(p)}$ is $c$-connected for all $p\ge 1$.

\end{lemma}

\begin{proof}

  Assume that the lemma is not correct and that $p$ is the smallest number, so that $(G,P)^{(p)}$ is not $c$-connected. So $p> \frac{c-1}{c_1}$. This implies that
  there is a cut set $C$ of size at most $c-1$ and that there is at least one blok, w.l.o.g.\ $G_{p-1}$ that contains fewer than $c_1$ vertices of $C$. So all
  vertices in $V(G_{p-1})\setminus C$ belong to the same connected component $K$.  We will prove that also the other vertices of $(G,P)^{(p)}$ that are not in $C$
  belong to $K$.  Let $B= \{(p-2,w_1),\dots ,(p-2,w_r),(0,v_1),\dots ,(0,v_r)\}$.

  First, we discuss $v\in V((G,P)^{(p-1)})\setminus C$, but $v  \not\in B$.  As $(G,P)^{(p-1)}$ is $c$-connected, as a
  consequence of Menger's theorem, there are at least $c$ paths to $B$ that are vertex disjoint up to
  $v$.  W.l.o.g.\ we can assume that these paths contain vertices of $B$ only as endpoints, as otherwise we can just take the initial part to the first vertex
  in $B$. These paths are also paths in $(G,P)^{(p)}$, so by extending the paths with the edges between $G_0$, resp.\ $G_{p-2}$ and $G_{p-1}$ we get at least $c$
  (except for $v$ itself) vertex disjoint paths from $v$ to a vertex of $G_{p-1}$. So at least one of these paths does not contain a vertex of $C$ and as the endpoint of this path is in $K$, so is $v$.

  Now let $v\in V((G,P)^{(p-1)})\setminus C$, but $v \in B$. As $(G,P)^{(1)}$ is $c$-connected and $v$ is incident
  to only one edge in $(G,P)^{(1)}$ that is not in $G$, the closed neighbourhood $N_G[v]$ of $v$ in $G$ contains at least $c$ elements. As each vertex in $C\cap
  V(G_{p-1})$ has at most one neighbour outside $G_{p-1}$, at most $c-1$ of these vertices can be in $C$ or have a neighbour in $C\cap V(G_{p-1})$. So there is
  a vertex $w$ in $N_G[v]$ that is not in $C$ and has no neighbour in $C\cap V(G_{p-1})$. If $w\in B$, $w$ and therefore also $v$ are in the same component as the
  neighbour of $w$ in $G_{p-1}$ -- that is $K$. If $w\not\in B$, we have already proven that $w\in K$.

  In each case there is only one connected component, so $C$ is no cutset -- contradicting the assumption.
  
\end{proof}

For the explicit examples of cycle extension pairs, the necessary prerequisites for the connectivity claimed were verified by computer, but can also be easily checked by hand.

We will now give graphs that prove lower bounds for $\domfracCinf{\cal C}$ for some classes ${\cal C}$. These graphs were found by computer by first searching for
graphs $G'$ that can serve as $(G,P)^{(1)}$ and have a large ratio $\domfrac{G'}$. Then it was tested whether some of the edges can be assigned to $G$ and the others to
$P$ in a way that $(G,P)$ is a cycle extension pair proving a good bound for $\domfracCinf{}$ and achieving the desired connectivity.

\begin{figure}[htbp]
  \begin{center}
      \includegraphics[width=68mm]{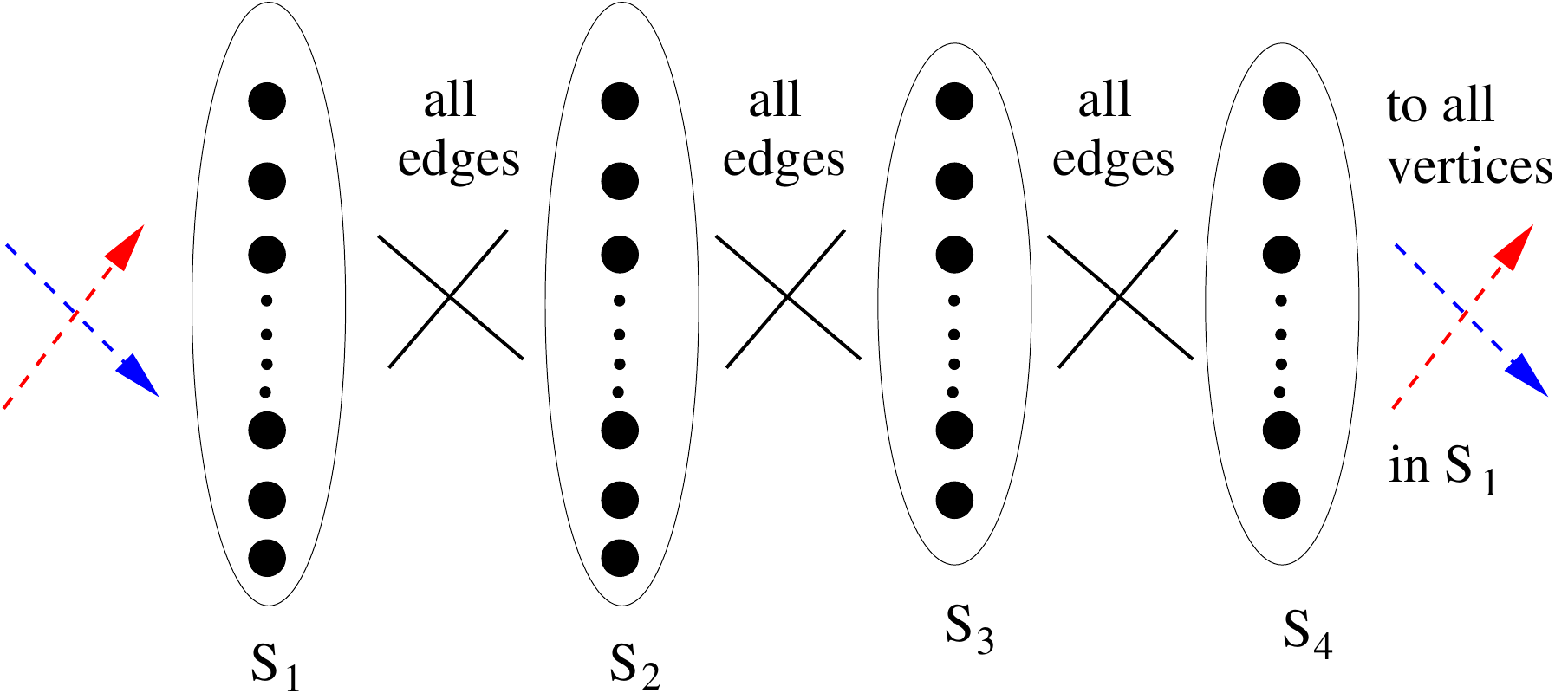}
\end{center}
\caption{A cycle extension pair $(G,P)$ for arbitrary $k\ge 3$. $G$ has $2k$ vertices distributed over four independent sets $S_1,\dots ,S_4$. $S_1$ and $S_2$
  have $\lceil \frac{k}{2}\rceil$ vertices, while $S_3$ and $S_4$ have $\lfloor \frac{k}{2}\rfloor$ vertices. In $G$, all edges between $S_1$ and $S_2$, $S_2$ and $S_3$,
  and $S_3$ and $S_4$ are present and $P$ is the set of all $(w,v)$ with $w\in S_4$, $v\in S_1$. Each $(G,P)^{(p)}$ is a $k$-regular, bipartite graph. For even $k$ (or
  $p=1$), $(G,P)^{(p)}$ is $k$-connected. For odd $k$ and $p>1$, $(G,P)^{(p)}$ is only $(k-1)$-connected. }\label{fig:kconstr}
\end{figure}

\begin{lemma}\label{lem:limitsdomfracreg}

  For each $k\ge 3$ and ${\cal C}$ the class of $k$-regular, $(2\lfloor \frac{k}{2}\rfloor)$-connected graphs 
  -- or ${\cal C}$ the same class with the additional restriction of being bipartite  --
  we have $\frac{k}{3} \le \domfracCinf{{\cal C}} \le  \frac{k}{2}$.

\end{lemma}

\begin{proof}

  The upper bound $\frac{k}{2}$ is proven for general $k$-regular graphs in \cite{indepratiok2},
  but with $\beta=\frac{k-1}{2k-1}$ and Corollary~\ref{cor:domfrac} even slightly better (and more complicated looking) bounds follow.
  To this end we only have to prove the lower bound.

  The lower bound is trivial for $k=3$, so assume that $k\ge 4$.
  We will use the graphs described in Figure~\ref{fig:kconstr}. For even $k$ they are known as special {\em $\frac{k}{2}$-expansions} of the cycle and studied e.g.\ in \cite{indepdom_regular}.
  It can be easily seen that they are $(2\lfloor \frac{k}{2}\rfloor)$-connected, as for at least two sets corresponding to one of the $S_i$, all vertices must be deleted in order to disconnect that graph.
    Let $G_i$ be one of the copies of $G$ in a graph
  $(G,P)^{(p)}$. Then the vertices in $S_2$ and $S_3$ cannot be dominated from outside the copy, and as there is no vertex dominating $S_2\cup S_3$, each dominating set
  contains at least two vertices of $G_i$. On the other hand, one vertex from $S_2$ and one vertex from $S_3$ dominate the whole of $G_i$, so $\dom{(G,P)^{(p)}}=2p$.

  For $\idom{(G,P)^{(p)}}$, note that if a set $S_i$ of a copy of $G$ contains a vertex $v$ of an independent dominating set $I$, then $S_i\subseteq I$, as all vertices in
  the neighbouring sets are excluded by $v$, so that the other vertices in $S_i$ can only be dominated by themselves. We choose $p$ a multiple of $3$, so we can
  rename the different sets $S_i$ in the copies of $G$ as $T_0,\dots ,T_{4p-1}$ and split them into $\frac{4p}{3}$ parts $(T_0,T_1,T_2)$, $(T_3,T_4,T_5)$, \dots.
  In each of the parts at least one of the $T_i$ must be contained in $I$, but choosing $T_1,T_4,\dots $ gives a dominating independent set, so we get
  $\idom{(G,P)^{(p)}}\le \frac{4p}{3}\cdot \lceil \frac{k}{2}\rceil$ in each case and  $\idom{(G,P)^{(p)}}= \frac{2kp}{3}$ if $k$ is even, as all parts have the same size $\lceil \frac{k}{2}\rceil=  \frac{k}{2}$.

  In case $k$ is odd, it is also not difficult to prove the exact bound. Choosing again all $T_{i+3h}$ for fixed $i$ and all $h$ with the index modulo $4p$ results in 
  $\frac{4p}{6}$ sets of size $\lfloor \frac{k}{2}\rfloor$ and $\frac{4p}{6}$ sets of size $\lceil \frac{k}{2}\rceil$, so again we get
  $\idom{(G,P)^{(p)}}\le  \frac{2kp}{3}$. On the other hand, we can assume that $T_0$ has size $\lfloor \frac{k}{2}\rfloor$, $T_1$ has size
$\lceil \frac{k}{2}\rceil$, and we   can split $T_0,\dots ,T_{4p-1}$ into segments $(T_0,\dots ,T_{11})$, $(T_{12},\dots ,T_{23})$, \dots
  each containing twelve sets. Even assuming that the first and last set in a segment can be dominated from the outside, it can be checked by going through a small number of cases
  that each independent dominating set contains $2$
  sets of size $\lceil \frac{k}{2}\rceil$ and at least $2$ sets of size $\lfloor \frac{k}{2}\rfloor$  of the segment, one set of size $\lceil \frac{k}{2}\rceil$ and at least $4$
  sets of size $\lfloor \frac{k}{2}\rfloor$, or  three
  sets of size $\lceil \frac{k}{2}\rceil$ and at least $1$ set of size $\lfloor \frac{k}{2}\rfloor$. The last case is obviously not optimal and as $k\ge 4$

  \[ \frac{4(k-1)}{2}+\frac{k+1}{2}=\frac{5k-3}{2}\ge \frac{4k}{2} \]

  we can conclude that each segment contains at least $2k$ vertices of the dominating set, so $\idom{(G,P)^{(p)}}\ge 2k\cdot \frac{4p}{12}=\frac{2kp}{3}$ and finally

  \[\domfrac{(G,P)^{(p)}} = \frac{2kp/3}{2p}= \frac{k}{3}.\]

\end{proof}

To improve our bound for planar cubic graphs, we still need one short lemma before the final theorem. It allows us to replace a bounded part of each graph in an inﬁnite
sequence of toroidal graphs so that the resulting graphs are planar and the asymptotic ratio is preserved.

\begin{lemma}\label{lem:cut}

  Let $G_1,G_2,\dots $ be an infinite sequence of graphs with $\idom{G_j}=j\cdot c_{\gamma_i}$ and $\dom{G_j}=j\cdot c_{\gamma}$. Furthermore, let
  $G'_1,G'_2,\dots $ be an infinite sequence of graphs that all contain the same finite induced subgraph $S$ such that there is a constant $q$ so that for each $j$
  the graph $G'_j$ can be obtained from $G_j$ by removing $q$ edges and connecting vertices of the disjoint graph $S$ with each of the vertices that lost an edge in
  $G_j$. Assume that all $G_j$ and $G'_j$ have maximum degree at most $k$.

  Then $\liminf_{j\to \infty} \domfrac{G'_j}\ge \domfrac{G_1}$.

\end{lemma}

\begin{proof}

  Let $V_{j,S}$ be the set of at most $2q$ vertices in $G_j$ that are connected to $S$ when constructing $G'_j$. We have $\idom{G_j}\le \idom{G'_j}+2qk$ as an independent dominating set of $G_j$ can be constructed by taking the intersection of any
  minimum independent dominating set of $G'_j$ with $G_j\setminus V_{j,S}$ and filling it up with a maximal independent set of the vertices in $N[V_{j,S}]$ that are not yet dominated. As for each $v\in V_{j,S}$ at most $k$ vertices
  can be in that independent set, the bound follows.

  On the other hand, $\dom{G'_j}\le \dom{G_j}+|S|$, as a dominating set of $G_j$ together with all vertices of $S$ is a dominating set of $G'_j$. So with $t_0=\max\{|S|,2qk\}$ we get

\begin{align*}
   \domfrac{G'_j} & \ge  \frac{\idom{G_j}-t_0}{\dom{G_j}+t_0}=\frac{\idom{G_j}}{\dom{G_j}+t_0}-\frac{t_0}{\dom{G_j}+t_0}=\frac{\dom{G_j}}{\dom{G_j}+t_0}\cdot \frac{\idom{G_j}}{\dom{G_j}}-\frac{t_0}{\dom{G_j}+t_0}\\
   & = \frac{j\cdot c_{\gamma}}{j\cdot c_{\gamma}+t_0}\cdot \domfrac{G_j}-\frac{t_0}{j\cdot c_{\gamma}+t_0} \\
\end{align*}

As $\domfrac{G_j}=\domfrac{G_1}$ for all $j$, for $j\to \infty$ this converges to $\domfrac{G_1}$.

  \end{proof}

For the following tables, upper bounds for a certain class $\cal C$ essentially mean that there are no graphs with a certain property in this class.
This implies that they are also valid 
for each subset of the class. So upper bounds for $\domfracC{\cal C}$ are also upper bounds for $\domfracCinf{\cal C}$.
Lower bounds -- which come from graphs with these properties -- extend to supersets of the class, but not necessarily to subsets.

\begin{theorem}

~
  
    For planar graphs we have the following bounds:

        \bigskip
	\begin{center}
        \begin{tabular}{|c|c|c|c|}
                \toprule
        ${\cal C}$ & $\domfracCinf{\cal C} \ge $ & $\domfracC{\cal C} \ge $ & $\domfracCinf{\cal C}, \domfracC{\cal C} \le $\\
      \midrule
          cubic polyhedra &  \multirow{2}{*}{$\displaystyle\frac{6}{5}$}   &  \multirow{2}{*}{$\displaystyle\frac{6}{5}$}   & \multirow{2}{*}{$\displaystyle\frac{9}{7}$}   \\
          $|V|>10$ &    &    &    \\
          \midrule
                   cubic bipartite &  \multirow{2}{*}{$\displaystyle\frac{7}{6}$}   &  \multirow{2}{*}{$\displaystyle\frac{7}{6}$}   & \multirow{2}{*}{$\displaystyle\frac{9}{7}$}   \\
          polyhedra &    &    &    \\
          \midrule
          cubic, planar, &  \multirow{2}{*}{$\displaystyle\frac{5}{4}$}   &  \multirow{2}{*}{$\displaystyle\frac{5}{4}$}   & \multirow{2}{*}{$\displaystyle\frac{9}{7}$}   \\
          2-connected, $|V|>10$ &    &    &    \\
          \midrule
          quartic polyhedra &  \multirow{2}{*}{$\displaystyle\frac{6}{5}$}   &  \multirow{2}{*}{$\displaystyle\frac{4}{3}$}   & \multirow{2}{*}{$\displaystyle\frac{9}{5}$}   \\
          4-connected &    &    &    \\
          \midrule
         quartic, planar &  \multirow{2}{*}{$\displaystyle\frac{4}{3}$}   &  \multirow{2}{*}{$\displaystyle\frac{4}{3}$}   & \multirow{2}{*}{$\displaystyle\frac{9}{5}$}   \\
          2-connected &    &    &    \\
          \midrule

    \end{tabular}
  	\end{center}

        For graphs with no restriction on the genus, we have the following bounds:

        \nopagebreak[4]
  \bigskip

	\begin{center}
          \begin{tabular}{|c|c|c|c|}
                  \toprule
        ${\cal C}$ & $\domfracCinf{\cal C} \ge $ & $\domfracC{\cal C} \ge $ & $\domfracCinf{\cal C}, \domfracC{\cal C} \le $\\
      \midrule
          cubic, 3-connected, & \multirow{2}{*}{${\displaystyle\frac{5}{4}}$}   & \multirow{2}{*}{$\displaystyle\frac{9}{7}$}   & \multirow{2}{*}{$\displaystyle\frac{4}{3}$}   \\
          $|V|>10$ &    &    &    \\
          \midrule
          cubic, no 4-cycles, &  \multirow{2}{*}{$\displaystyle\frac{7}{6}$}   &  \multirow{2}{*}{$\displaystyle\frac{7}{6}$}   & \multirow{2}{*}{$\displaystyle\frac{5}{4}$}   \\
          2-connected &    &    &    \\
            \midrule
          cubic, bipartite &  \multirow{2}{*}{$\displaystyle\frac{5}{4}$}   &  \multirow{2}{*}{$\displaystyle\frac{5}{4}$}   & \multirow{2}{*}{$\displaystyle\frac{4}{3}$}   \\
          3-connected, $|V|>10$ &    &    &    \\
          \midrule
           quartic, &  \multirow{2}{*}{$\displaystyle\frac{3}{2}$}   &  \multirow{2}{*}{$\displaystyle\frac{5}{3}$}   & \multirow{2}{*}{$\displaystyle\frac{9}{5}$}   \\
          4-connected, $|V|>8$ &    &    &    \\
          \midrule
            quartic, bipartite &  \multirow{2}{*}{$\displaystyle\frac{3}{2}$}   &  \multirow{2}{*}{$\displaystyle\frac{3}{2}$}   & \multirow{2}{*}{$\displaystyle\frac{9}{5}$}   \\
          4-connected, $|V|>8$ &    &    &    \\
          \midrule
             5-regular, &  \multirow{2}{*}{$\displaystyle\frac{27}{14}$}   &  \multirow{2}{*}{$\displaystyle\frac{27}{14}$}   & \multirow{2}{*}{$\displaystyle\frac{16}{7}$}   \\
          5-connected, $|V|>10$ &    &    &    \\
          \midrule
            5-regular, bipartite & \multirow{2}{*}{$\displaystyle\frac{93}{56}$}   &  \multirow{2}{*}{$\displaystyle\frac{93}{56}$}   & \multirow{2}{*}{$\displaystyle\frac{16}{7}$}   \\
          5-connected, $|V|>10$ &    &    &    \\
          \midrule

    \end{tabular}
   	\end{center}
        \bigskip

\end{theorem}

\begin{proof}

  The numbers in the tables come from the bounds in Corollary~\ref{cor:harvest} and the graphs, graph sequences, and results explained in the captions of Figures~\ref{fig:cubic3con9_7} to \ref{fig:4regpolyhedron}
  and \ref{fig:cubicpolyhedron} to \ref{fig:5regbip}.

  The bounds for $\domfracCinf{\cal C}$ for the 2-connected planar cases also use Lemma~\ref{lem:cut}. Here, one set of $P$-edges connecting two copies of $G$
  must be removed and the vertices incident to these edges must be connected to two small, suitable graphs to cut the torus open. The fact that the graphs are
  2-connected can be proven in the same way as in the cyclic case, but is -- just like planarity -- in fact obvious in these cases.

\end{proof}

\subsection*{Final remarks}

We have given upper and lower bounds for $\domfracC{\cal C}$ for several classes ${\cal C}$ and developed some theorems and programs that will hopefully turn out to be useful
for further research. Nevertheless, it is striking that (at the moment) the lower and upper bounds do not coincide for any of these classes.
 Nevertheless, if the result
 in \cite{38forcubicatleast12} -- that for $|V|>10$ the bound $\beta=\frac{3}{8}$ holds without excluding $K_{2,3}$ as a subgraph -- is confirmed, the graph in Figure~\ref{fig:cubic3con9_7}
 shows that the lower bound for $\domfracC{\cal C}$ is sharp for cubic $3$-connected graphs with more than $10$ vertices.

For some classes, sharp bounds for
$\frac{\idom{G}}{|V|}$ are known, but these do not necessarily imply sharp bounds for $\domfrac{G}$. For example, the graphs in the family ${\cal G}_{cubic}$ defined in
\cite{indepdom_pl_cubic} have an optimal value of  $\frac{\idom{G}}{|V|}=\frac{3}{8}$ in the class of $2$-connected cubic planar graphs on at least $12$ vertices.
 The family ${\cal G}'_{cubic}$ of planar modifications of the graphs in the family shown in
Figure~\ref{fig:cubic2conplane} is part of the same class, and the graphs in ${\cal G}'_{cubic}$ have a strictly smaller value of $\frac{\idom{G}}{|V|}$ than those in ${\cal G}_{cubic}$.
 Nevertheless, they have a larger value of $\domfrac{G}$.

While often small examples are found with a value for $\domfrac{G}$ that cannot be improved for many larger orders, which might suggest optimality, the 
examples for 2-connected planar graphs show that such a conclusion is premature: here the larger graphs have higher values of $\domfrac{G}$, and it is not even known whether the
supremum is ever realised by a graph in the class.  

The bounds for $\sigma_k$ determined in this article depend only on $k$ and do not consider further restrictions on the class -- like planarity or small, forbidden subgraphs.
Perhaps better values for $\domfracC{\cal C}$ can be obtained in some cases by taking
such restrictions into account.

Though we used computers when searching for cycle extension pairs, there was no real {\em complete search}, but rather some heuristic. We searched for graphs
with a large value of $\domfrac{G}$ and checked whether these graphs can be used to construct good cycle extension pairs. As $\domfrac{(G,P)^{(p)}}$ can increase as
well as decrease for $p>1$, this is just a heuristic and does not necessarily give optimal results. As cycle extension pairs are just one way to construct
infinite sequences, we think that a time-consuming complete search would probably not be justified.

There might also be interesting classes ${\cal C}$ where $\domfracCinf{\cal C}$ or even $\domfracC{\cal C}$ is $1$. For example, for $G$ in the class of cubic
polyhedra with maximum face size 6 -- including the chemically interesting class of fullerenes -- computations suggest that
$\idom{G}-\dom{G}\le 1$, which would imply $\domfracCinf{\cal C}=1$. For their duals -- that is, plane triangulations with maximum degree 6 -- it even seems that
$\idom{G}=\dom{G}$, which would imply $\domfracC{\cal C}=1$. This might be explained by the finite number of deviations from the hexagonal lattice or the triangular
lattice. Nevertheless, our computations were only for cubic polyhedra with maximum face size 6 up to 128 vertices and triangulations with maximum degree 6 up to 85 vertices and it is not clear
whether these results for small orders might be misleading.

Due to the nature of the cycle extension method, the polyhedral families we constructed from planar graphs all contain graphs with two faces with arbitrarily large size. To
show that this property is not essential, we also constructed two families of cubic polyhedra with bounded face size for which $\domfrac{G}>1$. They are
constructed from the toroidal families shown in Figures~\ref{fig:65bounded} and \ref{fig:fc7}. The first family of polyhedra contains relatively large faces but has a larger domination ratio than the second,
while the second has maximum face size 7 but a smaller domination ratio.  In each case we start from a genus $1$ family, cut open the torus at the dashed edges of one block, and add
two vertices of degree 3 to create a planar graph.  This shows that our previous remark about cubic polyhedra with maximum face size 6 cannot be extended to
face size 7 or larger. Since there are only finitely many cubic polyhedra with maximum face size 5 or smaller, this leaves maximum face size 6 as the only open case.

\medskip

The programs for computing the domination number and the independent domination number were tested by comparing the results against those of an independently developed program.
We tested all cubic graphs on up to 26 vertices, all quartic graphs on up to 18 vertices, and several other classes of graphs. For details, see the header of the programs. The results of the program computing $W(H)$ were tested by explicitly constructing
the graphs $(G,P)^{(p)}$ for the given cycle lengths and computing $\dom{(G,P)^{(p)}}$ and $\idom{(G,P)^{(p)}}$.

\section*{Use of AI}
The Gemini Deep Think AI contributed a proof idea for an earlier, weaker version of Theorem~\ref{thm:delta} that did not use the chromatic number. Claude was used
during the revision of the text to improve the language and check proofs and notation. Claude did not contribute to proof methods or ideas in this text.

\section*{Acknowledgements}
We thank Carol Zamfirescu for interesting initial discussions on this topic.

\bigskip

\begin{figure}[htbp]
\begin{center}
  \includegraphics[width=85mm]{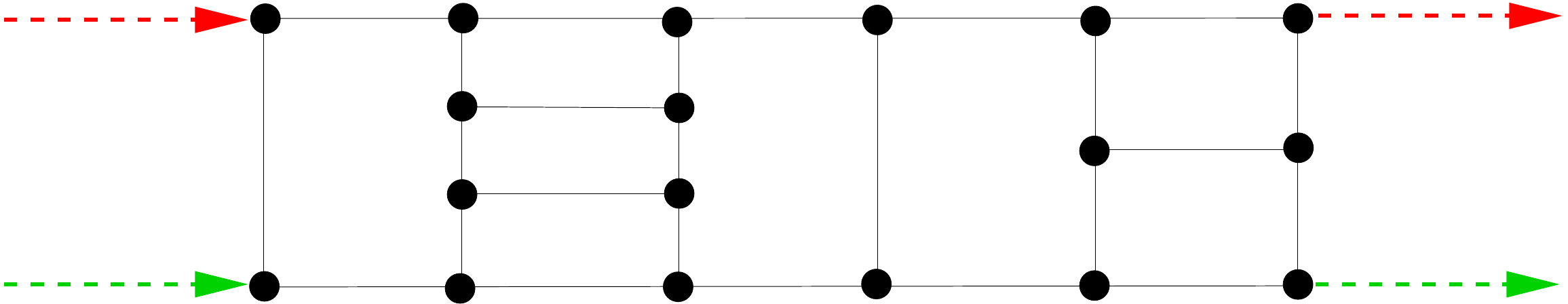}
\end{center}
\caption{A cycle extension pair $(G,P)$ -- with the set $P$ of directed edges given as coloured, dashed arrows -- for which $(G,P)^{(p)}$ is a cubic polyhedron with $\domfrac{(G,P)^{(p)}}= \frac{6}{5}$ for each $p\ge 1$. $W(H^{(G,P)}_{\gamma_i})=6$ and $W(H^{(G,P)}_{\gamma})=5$, both realised for a cycle of length $1$.}\label{fig:cubicpolyhedron}
\end{figure}

\begin{figure}[htbp]
\begin{center}
  \includegraphics[width=75mm]{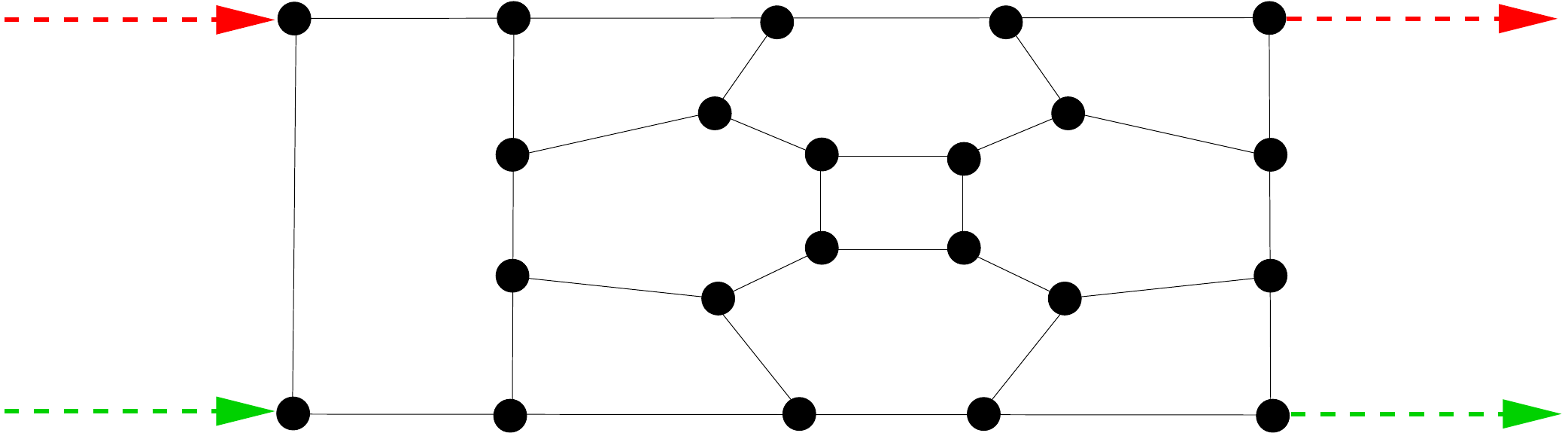}
\end{center}
\caption{A cycle extension pair $(G,P)$ for which $(G,P)^{(p)}$ is a cubic polyhedron with $\domfrac{(G,P)^{(p)}}= \frac{7}{6}$ for each $p\ge 1$. $W(H^{(G,P)}_{\gamma_i})=7$ and $W(H^{(G,P)}_{\gamma})=6$, both realised for a cycle of length $1$. For even $p$, $(G,P)^{(p)}$ is also bipartite.}\label{fig:cubicbippolyhedron}
\end{figure}

\begin{figure}[htbp]
\begin{center}
  \includegraphics[width=70mm]{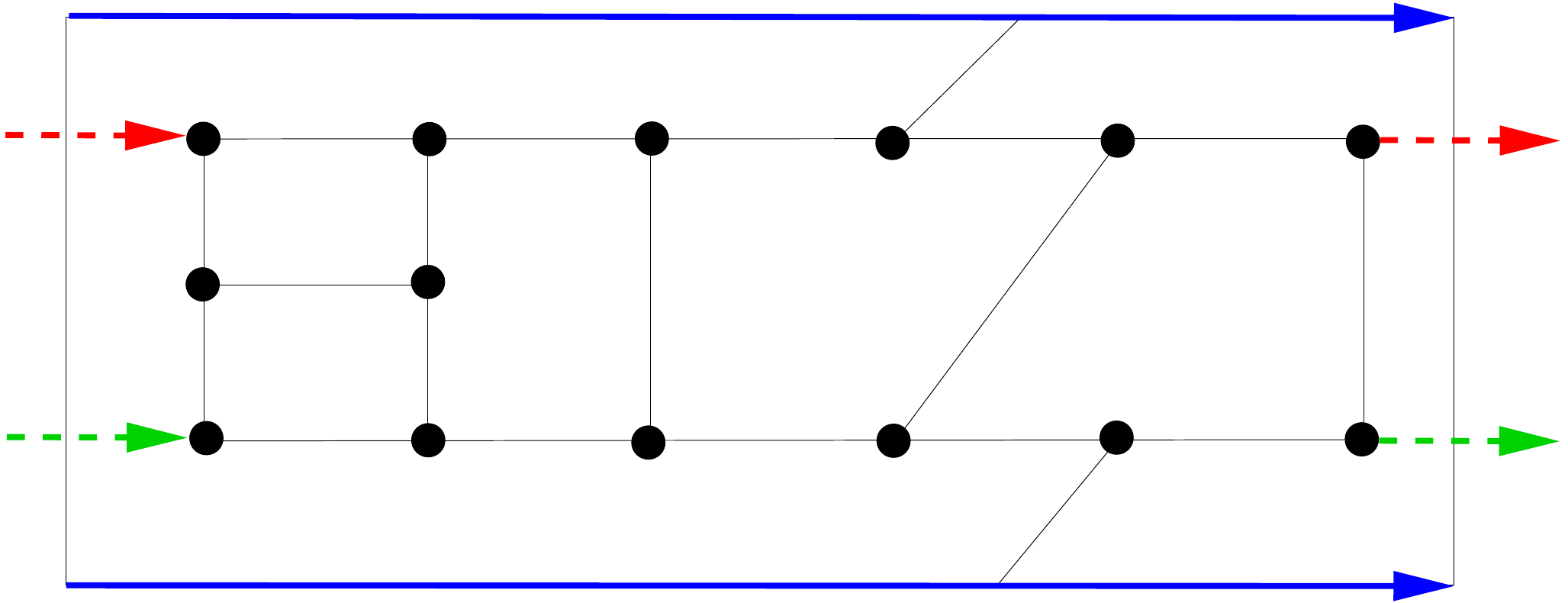}
\end{center}
\caption{A cycle extension pair $(G,P)$ for which $(G,P)^{(p)}$ is a cubic, 3-connected graph of genus 1 with $\domfrac{(G,P)^{(p)}}= \frac{5}{4}$ for each $p\ge 1$. $W(H^{(G,P)}_{\gamma_i})=5$ and $W(H^{(G,P)}_{\gamma})=4$, both realised for a cycle of length $1$.
The two blue arrows need to be identified to form a tube which the cyclic construction closes to a torus. }\label{fig:cubic2conplane}
\end{figure}

\begin{figure}[htbp]
  \begin{center}
    \resizebox{0.45\textwidth}{!}{
      \input{3reg_girth6_frac_7_6.tikz}
      }
\end{center}
\caption{A cycle extension pair $(G,P)$ for which $(G,P)^{(p)}$ is cubic, 2-connected and of girth 6, and satisfies $\domfrac{(G,P)^{(p)}}= \frac{7}{6}$ for each $p\ge 1$.
   $W(H^{(G,P)}_{\gamma_i})=7$ and $W(H^{(G,P)}_{\gamma})=6$, both realised for a cycle of length $1$. The genus of $(G,P)^{(p)}$ increases with $p$.}\label{fig:cubicbip}
\end{figure}

\begin{figure}[htbp]
  \begin{center}
    \resizebox{0.32\textwidth}{!}{
      \input{3reg_bip_frac_5_4_edges_9_5_7_11.tikz}
      }
\end{center}
\caption{A cycle extension pair $(G,P)$ for which $(G,P)^{(p)}$ is cubic, bipartite, 3-connected, and $\domfrac{(G,P)^{(p)}}= \frac{5}{4}$ for each $p\ge 1$.
   $W(H^{(G,P)}_{\gamma_i})=5$ and $W(H^{(G,P)}_{\gamma})=4$, both realised for a cycle of length $1$. The genus of $(G,P)^{(p)}$ increases with $p$. }\label{fig:cubicgirth6}
\end{figure}

\begin{figure}[htbp]
   \begin{center}
      \includegraphics[width=55mm]{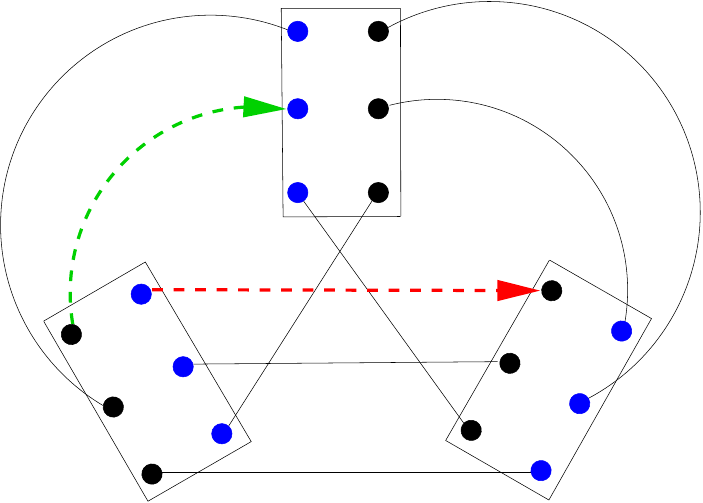}
\end{center}
   \caption{A cycle extension pair $(G,P)$ for which $(G,P)^{(p)}$ is 4-regular, bipartite, 4-connected, and $\domfrac{(G,P)^{(p)}}= \frac{3}{2}$ for each $p\ge 1$.
    The boxes represent complete bipartite graphs $K_{3,3}$. $W(H^{(G,P)}_{\gamma_i})=6$ and $W(H^{(G,P)}_{\gamma})=4$, both realised for a cycle of length $1$.}\label{fig:quartic}
\end{figure}

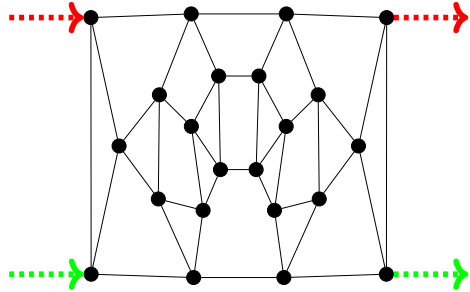
\begin{figure}[htbp]
    \begin{center}
    \resizebox{0.4\textwidth}{!}{
      \input{4reg_planar_4con_frac_6_5.tikz}
      }
\end{center}
  \caption{A cycle extension pair $(G,P)$ for which $(G,P)^{(p)}$ is planar, 4-regular, 4-connected, and $\domfrac{(G,P)^{(p)}}= \frac{6}{5}$ for each $p\ge 1$.
    $W(H^{(G,P)}_{\gamma_i})=6$ and $W(H^{(G,P)}_{\gamma})=5$, both realised for a cycle of length $1$.}\label{fig:quartic_plane}
\end{figure}

\begin{figure}[htbp]
\begin{center}
  \includegraphics[width=70mm]{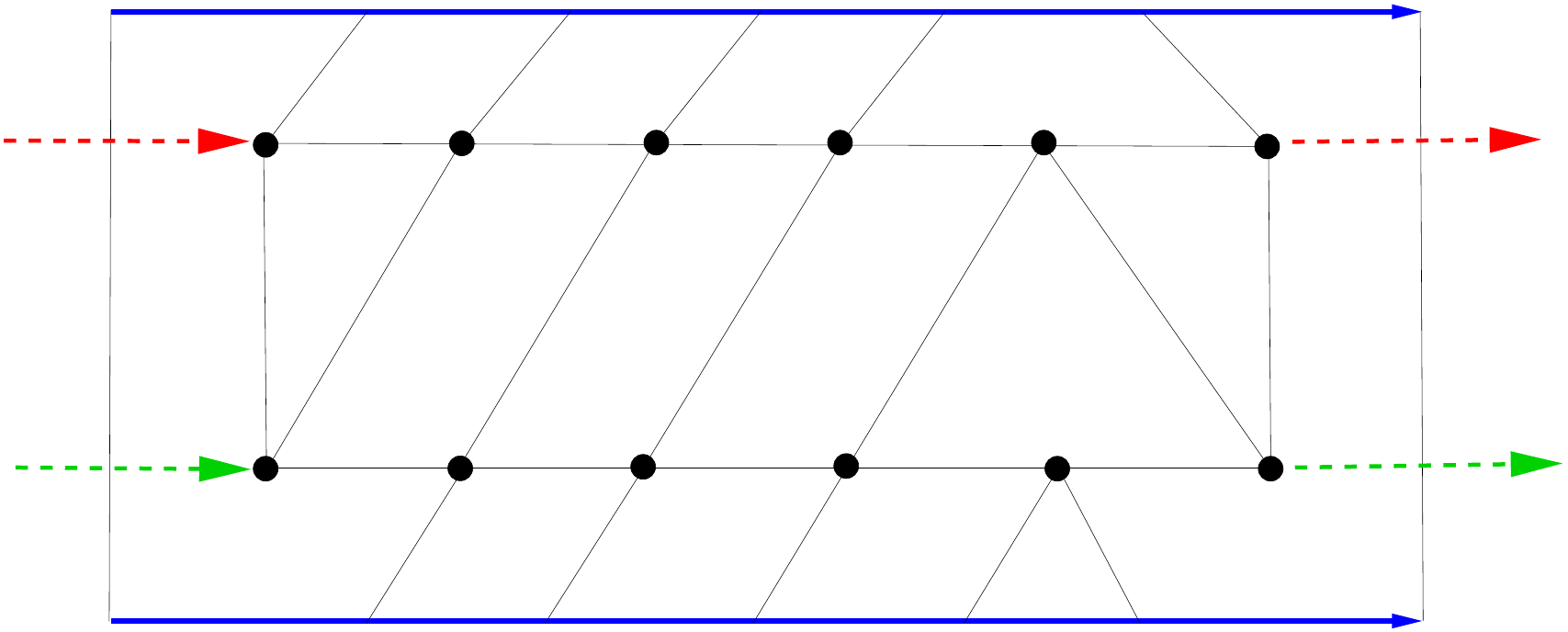}
\end{center}
\caption{A cycle extension pair $(G,P)$ for which $(G,P)^{(p)}$ is a quartic, 4-connected graph of genus 1 with $\domfrac{(G,P)^{(p)}}= \frac{4}{3}$ for each $p\ge 1$. $W(H^{(G,P)}_{\gamma_i})=4$ and $W(H^{(G,P)}_{\gamma})=3$, both realised for a cycle of length $1$. }\label{fig:quarticplane}
\end{figure}

\begin{figure}[htbp]
  \begin{center}
      \includegraphics[width=55mm]{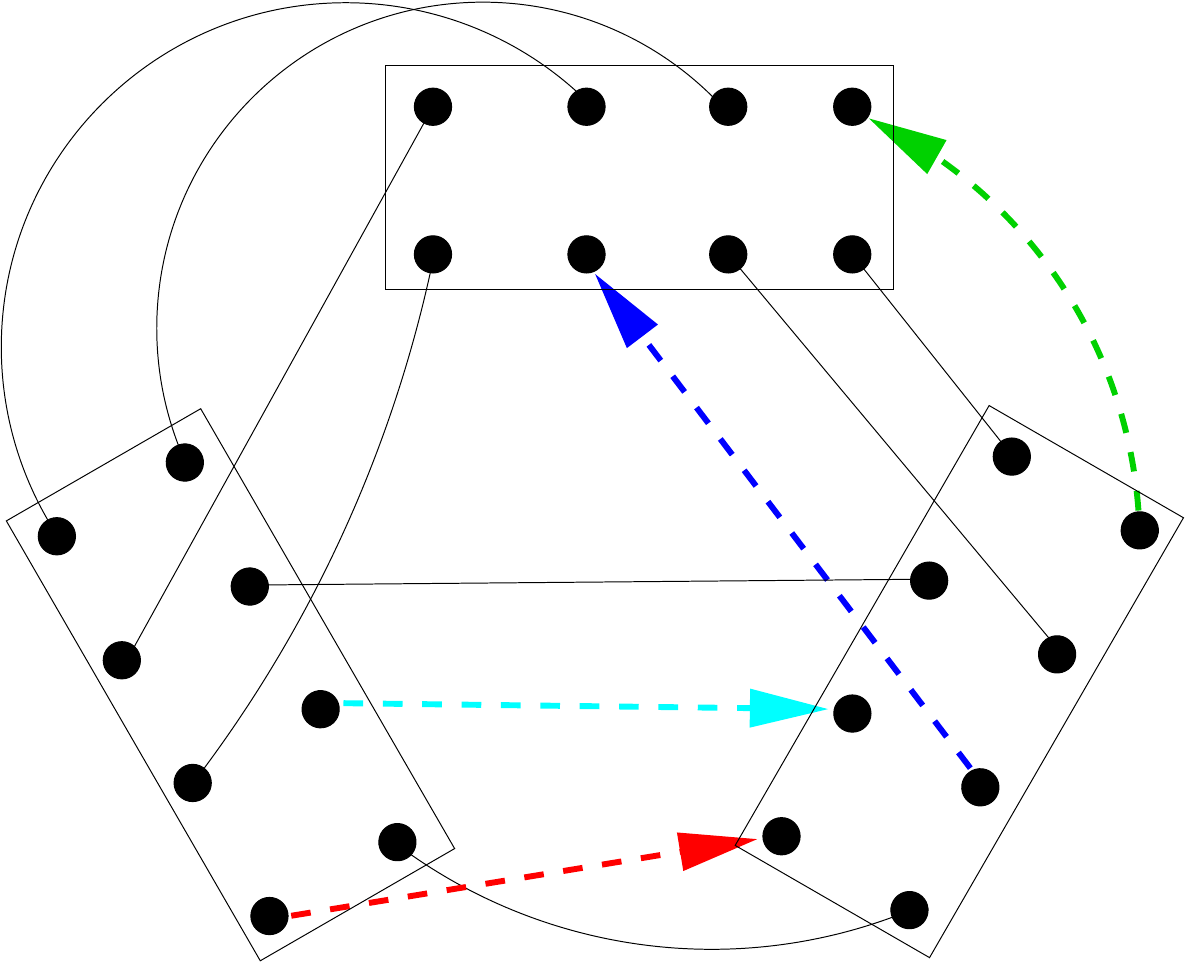}
\end{center}
\caption{A cycle extension pair $(G,P)$ with $\domfrac{(G,P)^{(1)}}=\frac{9}{5}$ for which $(G,P)^{(p)}$ is 5-regular,
  5-connected, and $\domfrac{(G,P)^{(p)}}= \frac{27}{14}$ for each $p$ that is a multiple of $3$.  The boxes represent complete bipartite graphs $K_{4,4}$. $W(H^{(G,P)}_{\gamma_i})=9$ and $W(H^{(G,P)}_{\gamma})=\frac{14}{3}$, realised with
  cycle length $1$ for $\gamma_i$ and $3$ for $\gamma$.}\label{fig:5r4g}
\end{figure}

\begin{figure}[htbp]
  \begin{center}
      \includegraphics[width=55mm]{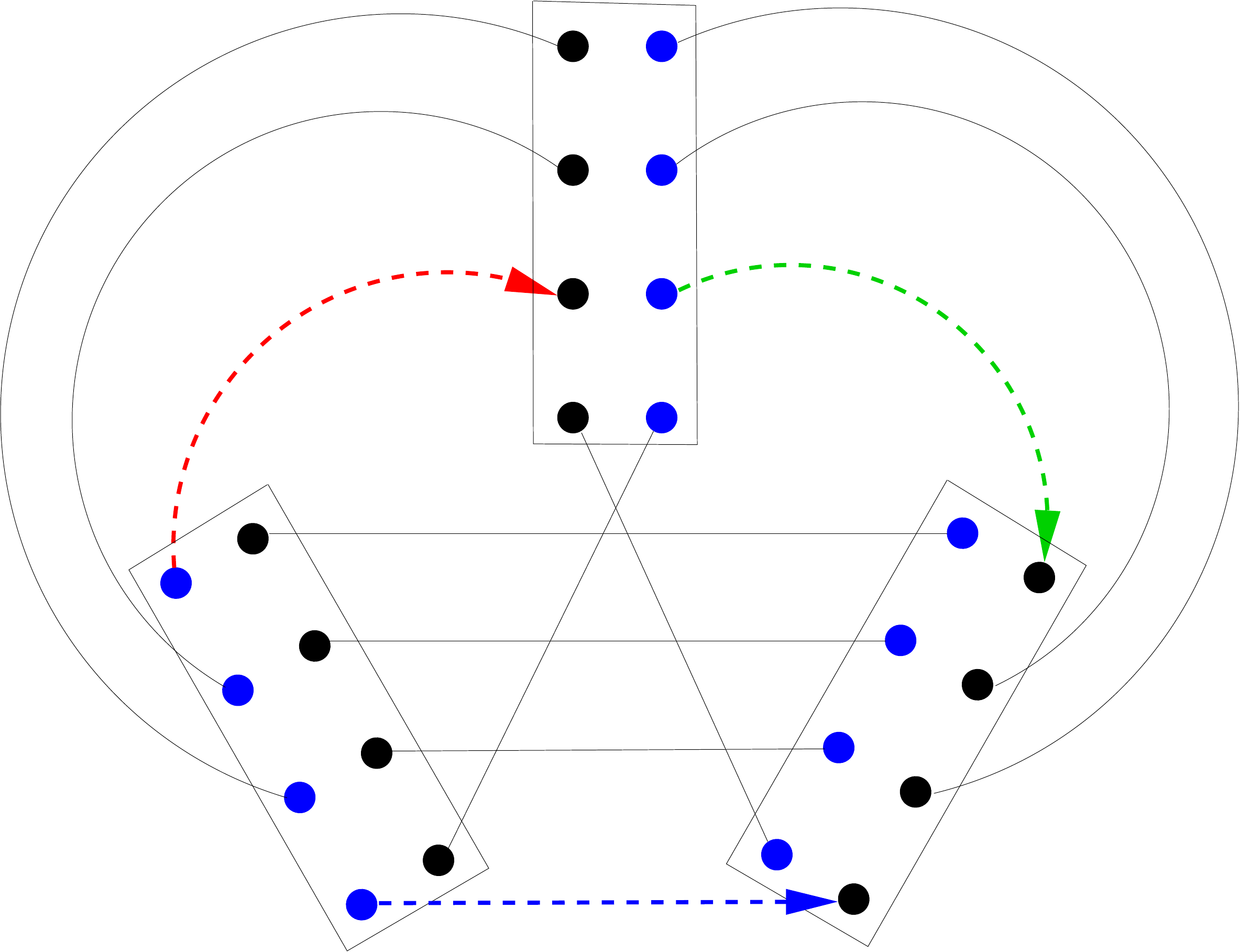}
\end{center}
\caption{A cycle extension pair $(G,P)$ with $\domfrac{(G,P)^{(1)}}=\frac{8}{5}$ for which $(G,P)^{(p)}$ is 5-regular, bipartite,
  5-connected, and $\domfrac{(G,P)^{(p)}}\ge \frac{93}{56}$ for each $p$ that is a multiple of $3$.  The boxes represent complete bipartite graphs $K_{4,4}$. $W(H^{(G,P)}_{\gamma_i})=\frac{31}{4}$ and $W(H^{(G,P)}_{\gamma})=\frac{14}{3}$, realised with
  cycle length $4$ for $\gamma_i$ and $3$ for $\gamma$.}\label{fig:5regbip}
\end{figure}

\begin{figure}[htbp]
	\begin{center}
		\includegraphics[width=65mm]{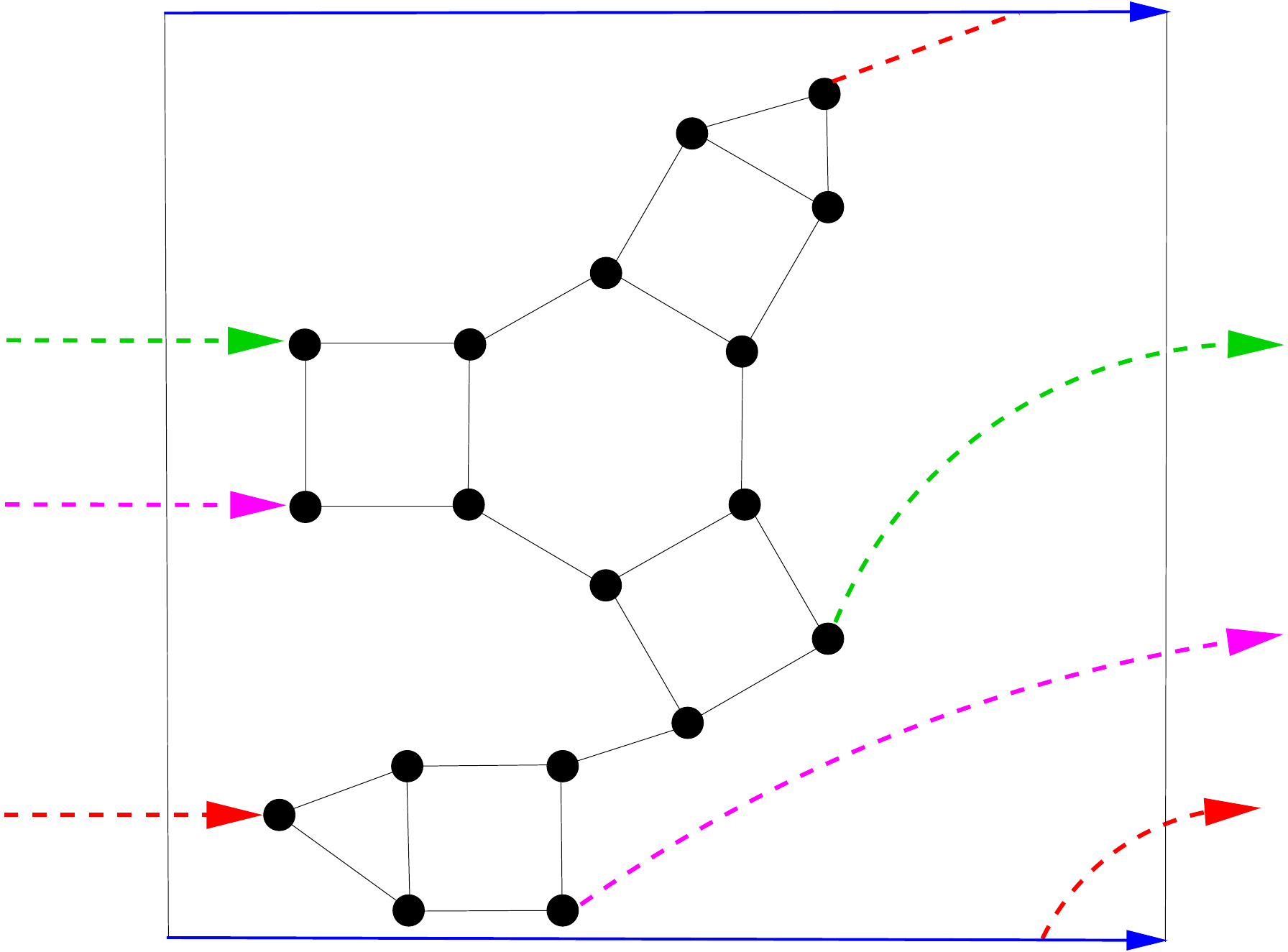}
	\end{center}
	\caption{A cycle extension pair $(G,P)$ for which $(G,P)^{(p)}$ is cubic, 3-connected, has genus 1, and $\domfrac{(G,P)^{(p)}}= \frac{6}{5}$.  $W(H^{(G,P)}_{\gamma_i})=6$ and $W(H^{(G,P)}_{\gamma})=5$, realised for a cycle of length $1$. Every face of $(G,P)^{(p)}$ is contained in at most 3 consecutive blocks.}\label{fig:65bounded}
\end{figure}

\begin{figure}[htbp]
	\begin{center}
		\includegraphics[width=65mm]{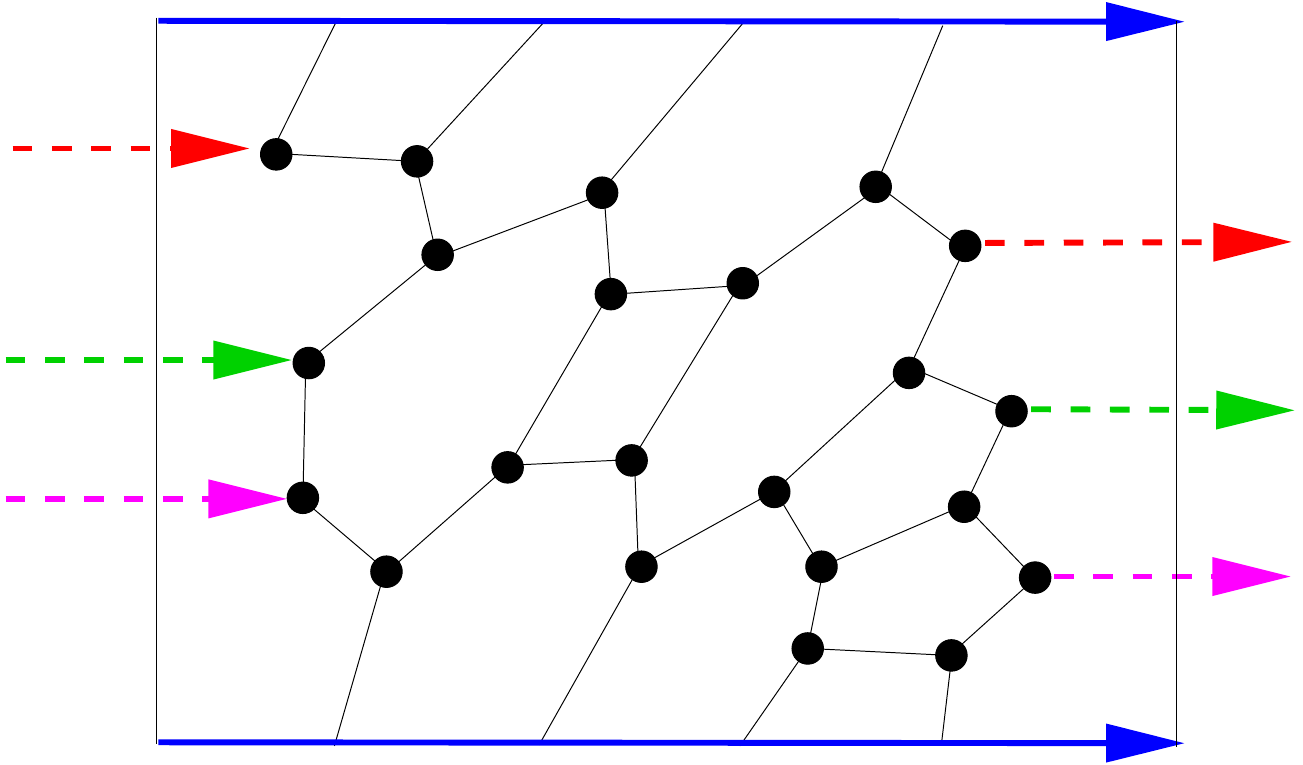}
	\end{center}
	\caption{A cycle extension pair $(G,P)$ for which $(G,P)^{(p)}$ is a cubic, 3-connected graph of genus 1 and maximum face size 7, and for which $\domfrac{(G,P)^{(p)}}= \frac{13}{12}$ for each even $p$.  $W(H^{(G,P)}_{\gamma_i})=\frac{13}{2}$ and $W(H^{(G,P)}_{\gamma})=6$, realised with cycle length $2$ for $\gamma_i$ and $1$ for $\gamma$. Every face of $(G,P)^{(p)}$ is in at most 2 consecutive blocks.}\label{fig:fc7}
\end{figure}

\FloatBarrier

\begingroup
\let\clearpage\relax


\endgroup

\end{document}

%% file: planar_4reg_4con_frac_4_3.tikz
\begin{tikzpicture}[scale=0.07]
\def\vertexscale{0.90}
\node [circle,black,draw,scale=\vertexscale,fill=black] (1) at (58.77853,-80.90170) {};
\node [circle,black,draw,scale=\vertexscale,fill=black] (2) at (-58.77853,-80.90170) {};
\node [circle,black,draw,scale=\vertexscale,fill=black] (3) at (95.10565,30.90170) {};
\node [circle,black,draw,scale=\vertexscale,fill=black] (4) at (24.59593,-63.83380) {};
\node [circle,black,draw,scale=\vertexscale,fill=black] (5) at (50.29040,-35.96474) {};
\node [circle,black,draw,scale=\vertexscale,fill=black] (6) at (-14.03464,-47.36718) {};
\node [circle,black,draw,scale=\vertexscale,fill=black] (7) at (-25.86934,15.96051) {};
\node [circle,black,draw,scale=\vertexscale,fill=black] (8) at (-0.00000,-6.87049) {};
\node [circle,black,draw,scale=\vertexscale,fill=black] (9) at (-17.73164,33.83667) {};
\node [circle,black,draw,scale=\vertexscale,fill=black] (10) at (-13.49236,49.63843) {};
\node [circle,black,draw,scale=\vertexscale,fill=black] (11) at (-0.00000,66.48333) {};
\node [circle,black,draw,scale=\vertexscale,fill=black] (12) at (13.49236,49.63843) {};
\node [circle,black,draw,scale=\vertexscale,fill=black] (13) at (0.00000,-31.28483) {};
\node [circle,black,draw,scale=\vertexscale,fill=black] (14) at (17.73164,33.83667) {};
\node [circle,black,draw,scale=\vertexscale,fill=black] (15) at (25.86934,15.96051) {};
\node [circle,black,draw,scale=\vertexscale,fill=black] (16) at (-23.53185,66.93602) {};
\node [circle,black,draw,scale=\vertexscale,fill=black] (17) at (-23.27919,-25.43487) {};
\node [circle,black,draw,scale=\vertexscale,fill=black] (18) at (-54.04413,37.67217) {};
\node [circle,black,draw,scale=\vertexscale,fill=black] (19) at (23.53185,66.93602) {};
\node [circle,black,draw,scale=\vertexscale,fill=black] (20) at (-52.27151,1.63786) {};
\node [circle,black,draw,scale=\vertexscale,fill=black] (21) at (-50.29040,-35.96474) {};
\node [circle,black,draw,scale=\vertexscale,fill=black] (22) at (54.04413,37.67217) {};
\node [circle,black,draw,scale=\vertexscale,fill=black] (23) at (-95.10565,30.90170) {};
\node [circle,black,draw,scale=\vertexscale,fill=black] (24) at (-0.00000,100.00000) {};
\node [circle,black,draw,scale=\vertexscale,fill=black] (25) at (23.27919,-25.43487) {};
\node [circle,black,draw,scale=\vertexscale,fill=black] (26) at (14.03464,-47.36718) {};
\node [circle,black,draw,scale=\vertexscale,fill=black] (27) at (52.27151,1.63786) {};
\node [circle,black,draw,scale=\vertexscale,fill=black] (28) at (-24.59593,-63.83380) {};
\draw [black] (1) to (2);
\draw [black] (1) to (4);
\draw [black] (1) to (5);
\draw [black] (1) to (3);
\draw [black] (2) to (23);
\draw [black] (2) to (21);
\draw [black] (2) to (28);
\draw [black] (3) to (27);
\draw [black] (3) to (22);
\draw [black] (3) to (24);
\draw [black] (4) to (28);
\draw [black] (4) to (26);
\draw [black] (4) to (5);
\draw [black] (5) to (25);
\draw [black] (5) to (27);
\draw [black] (6) to (28);
\draw [black] (6) to (17);
\draw [black] (6) to (13);
\draw [black] (6) to (26);
\draw [black] (7) to (20);
\draw [black] (7) to (18);
\draw [black] (7) to (9);
\draw [black] (7) to (15);
\draw [black] (8) to (17);
\draw [black] (8) to (20);
\draw [black] (8) to (27);
\draw [black] (8) to (25);
\draw [black] (9) to (14);
\draw [black] (9) to (18);
\draw [black] (9) to (10);
\draw [black] (10) to (16);
\draw [black] (10) to (11);
\draw [black] (10) to (12);
\draw [black] (11) to (12);
\draw [black] (11) to (16);
\draw [black] (11) to (19);
\draw [black] (12) to (19);
\draw [black] (12) to (14);
\draw [black] (13) to (26);
\draw [black] (13) to (17);
\draw [black] (13) to (25);
\draw [black] (14) to (22);
\draw [black] (14) to (15);
\draw [black] (15) to (27);
\draw [black] (15) to (22);
\draw [black] (16) to (18);
\draw [black] (16) to (24);
\draw [black] (17) to (21);
\draw [black] (18) to (23);
\draw [black] (19) to (24);
\draw [black] (19) to (22);
\draw [black] (20) to (21);
\draw [black] (20) to (23);
\draw [black] (21) to (28);
\draw [black] (23) to (24);
\draw [black] (25) to (26);
\end{tikzpicture}

%% file: 3reg_girth6_frac_7_6.tikz
  \begin{tikzpicture}[scale=0.065]
    \begin{scope}[rotate=-43]
\def\vertexscale{1.10}
\def\labelscale{1.20}
\node [circle,black,draw,scale=\vertexscale] (1) at (61.56076,-42.10936) {1};
\node [circle,black,draw,scale=\vertexscale] (2) at (-41.87656,26.61584) {2};
\node [circle,black,draw,scale=\vertexscale] (3) at (10.03398,45.93205) {3};
\node [circle,black,draw,scale=\vertexscale] (4) at (65.84590,29.88577) {4};
\node [circle,black,draw,scale=\vertexscale] (5) at (44.48535,-59.55224) {5};
\node [circle,black,draw,scale=\vertexscale] (6) at (-67.48053,31.85816) {6};
\node [circle,black,draw,scale=\vertexscale] (7) at (-8.42800,-3.19036) {7};
\node [circle,black,draw,scale=\vertexscale] (8) at (16.99362,21.88218) {8};
\node [circle,black,draw,scale=\vertexscale] (9) at (-28.45292,69.87143) {9};
\node [circle,black,draw,scale=0.9*\vertexscale] (10) at (-19.86918,47.44837) {10};
\node [circle,black,draw,scale=0.9*\vertexscale] (11) at (-61.08295,-14.15495) {11};
\node [circle,black,draw,scale=0.9*\vertexscale] (12) at (-36.81955,1.64223) {12};
\node [circle,black,draw,scale=0.9*\vertexscale] (13) at (26.83104,67.61925) {13};
\node [circle,black,draw,scale=0.9*\vertexscale] (14) at (-56.06836,-37.68204) {14};
\node [circle,black,draw,scale=0.9*\vertexscale] (15) at (-39.89803,-54.25939) {15};
\node [circle,black,draw,scale=0.9*\vertexscale] (16) at (-17.61457,-58.67273) {16};
\node [circle,black,draw,scale=0.9*\vertexscale] (17) at (66.27599,-18.33194) {17};
\node [circle,black,draw,scale=0.9*\vertexscale] (18) at (42.31962,15.51863) {18};
\node [circle,black,draw,scale=0.9*\vertexscale] (19) at (41.97645,-11.06112) {19};
\node [circle,black,draw,scale=0.9*\vertexscale] (20) at (18.47086,-35.57875) {20};
\node [circle,black,draw,scale=0.9*\vertexscale] (21) at (22.42063,-63.26458) {21};
\node [circle,black,draw,scale=0.9*\vertexscale] (22) at (-4.00181,-32.50024) {22};
\tkzDefPoint(-60.87614,79.33533){23}
\tkzDefPoint(-38.26834,-92.38795){24}
\tkzDefPoint(-79.33533,60.87614){25}
\tkzDefPoint(-13.05262,-99.14449){26}
\tkzDefPoint(13.05262,-99.14449){27}
\tkzDefPoint(-79.33533,-60.87614){28}
\tkzDefPoint(-60.87614,-79.33533){29}
\tkzDefPoint(-99.14449,-13.05262){30}
\tkzDefPoint(-99.14449,13.05262){31}
\tkzDefPoint(79.33533,-60.87614){32}
\tkzDefPoint(60.87614,-79.33533){33}
\tkzDefPoint(99.14449,13.05262){34}
\tkzDefPoint(99.14449,-13.05262){35}
\tkzDefPoint(60.87614,79.33533){36}
\tkzDefPoint(79.33533,60.87614){37}
\tkzDefPoint(-13.05262,99.14449){38}
\tkzDefPoint(13.05262,99.14449){39}
\tkzDefPoint(-38.26834,92.38795){40}
\tkzDefPoint(92.38795,-38.26834){41}
\tkzDefPoint(-92.38795,-38.26834){42}
\tkzDefPoint(38.26834,92.38795){43}
\tkzDefPoint(38.26834,-92.38795){44}
\tkzDefPoint(-92.38795,38.26834){45}
\tkzDefPoint(92.38795,38.26834){46}
\draw [->,red,dashed, line width=1.2mm] (1) to (5);
\draw [black] (1) to (17);
\draw [black] (1) to (32);
\node [draw=none,black,fill=none,scale=\labelscale] () at (83.30210,-63.91995) {9};
\draw [black] (2) to (6);
\draw [black] (2) to (10);
\draw [black] (2) to (12);
\draw [black] (3) to (10);
\draw [black] (3) to (13);
\draw [black] (3) to (8);
\draw [black] (4) to (18);
\draw [black] (4) to (37);
\node [draw=none,black,fill=none,scale=\labelscale] () at (83.30210,63.91995) {15};
\draw [black] (4) to (34);
\node [draw=none,black,fill=none,scale=\labelscale] () at (104.10171,13.70525) {6};
\draw [black] (5) to (33);
\node [draw=none,black,fill=none,scale=\labelscale] () at (63.91995,-83.30210) {6};
\draw [black] (5) to (21);
\draw [black] (6) to (25);
\node [draw=none,black,fill=none,scale=\labelscale] () at (-83.30210,63.91995) {5};
\draw [black] (6) to (31);
\node [draw=none,black,fill=none,scale=\labelscale] () at (-104.10171,13.70525) {4};
\draw [black] (7) to (8);
\draw [black] (7) to (22);
\draw [black] (7) to (12);
\draw [black] (8) to (18);
\draw [black] (9) to (23);
\node [draw=none,black,fill=none,scale=\labelscale] () at (-63.91995,83.30210) {1};
\draw [black] (9) to (38);
\node [draw=none,black,fill=none,scale=\labelscale] () at (-13.70525,104.10171) {16};
\draw [black] (9) to (10);
\draw [black] (11) to (12);
\draw [black] (11) to (14);
\draw [black] (11) to (30);
\node [draw=none,black,fill=none,scale=\labelscale] () at (-104.10171,-13.70525) {17};
\draw [black] (13) to (36);
\node [draw=none,black,fill=none,scale=\labelscale] () at (63.91995,83.30210) {14};
\draw [black] (13) to (39);
\node [draw=none,black,fill=none,scale=\labelscale] () at (13.70525,104.10171) {21};
\draw [black] (14) to (15);
\draw [black] (14) to (28);
\node [draw=none,black,fill=none,scale=\labelscale] () at (-83.30210,-63.91995) {13};
\draw [black] (15) to (29);
\node [draw=none,black,fill=none,scale=\labelscale] () at (-63.91995,-83.30210) {4};
\draw [black] (15) to (16);
\draw [black] (16) to (22);
\draw [black] (16) to (26);
\node [draw=none,black,fill=none,scale=\labelscale] () at (-13.70525,-104.10171) {9};
\draw [black] (17) to (19);
\draw [black] (17) to (35);
\node [draw=none,black,fill=none,scale=\labelscale] () at (104.10171,-13.70525) {11};
\draw [black] (18) to (19);
\draw [black] (19) to (20);
\draw [black] (20) to (21);
\draw [black] (20) to (22);
\draw [black] (21) to (27);
\node [draw=none,black,fill=none,scale=\labelscale] () at (13.70525,-104.10171) {13};
\tkzDefPoint(-35.47990,93.49426){A}
\tkzDefPoint(35.47990,93.49426){B}
\tkzDefPoint(0.0,0.0){C}
\tkzDrawArc[<-,line width=0.9mm, red](C,B)(A)
\tkzDefPoint(41.02235,91.19850){A}
\tkzDefPoint(91.19850,41.02235){B}
\tkzDefPoint(0.0,0.0){C}
\tkzDrawArc[<-,line width=0.9mm, green](C,B)(A)
\tkzDefPoint(93.49426,35.47990){A}
\tkzDefPoint(93.49426,-35.47990){B}
\tkzDefPoint(0.0,0.0){C}
\tkzDrawArc[<-,line width=0.9mm, orange](C,B)(A)
\tkzDefPoint(91.19850,-41.02235){A}
\tkzDefPoint(41.02235,-91.19850){B}
\tkzDefPoint(0.0,0.0){C}
\tkzDrawArc[<-,line width=0.9mm, blue](C,B)(A)
\tkzDefPoint(35.47990,-93.49426){A}
\tkzDefPoint(-35.47990,-93.49426){B}
\tkzDefPoint(0.0,0.0){C}
\tkzDrawArc[->,line width=0.9mm, red](C,B)(A)
\tkzDefPoint(-41.02235,-91.19850){A}
\tkzDefPoint(-91.19850,-41.02235){B}
\tkzDefPoint(0.0,0.0){C}
\tkzDrawArc[->,line width=0.9mm, green](C,B)(A)
\tkzDefPoint(-93.49426,-35.47990){A}
\tkzDefPoint(-93.49426,35.47990){B}
\tkzDefPoint(0.0,0.0){C}
\tkzDrawArc[->,line width=0.9mm, orange](C,B)(A)
\tkzDefPoint(-91.19850,41.02235){A}
\tkzDefPoint(-41.02235,91.19850){B}
\tkzDefPoint(0.0,0.0){C}
\tkzDrawArc[->,line width=0.9mm, blue](C,B)(A)
\node [black,circle,draw,fill=white,scale=0.75,line width=1mm] (24) at (-38.26834,-92.38795) {};
\node [black,circle,draw,fill=white,scale=0.75,line width=1mm] (40) at (-38.26834,92.38795) {};
\node [black,circle,draw,fill=white,scale=0.75,line width=1mm] (41) at (92.38795,-38.26834) {};
\node [black,circle,draw,fill=white,scale=0.75,line width=1mm] (42) at (-92.38795,-38.26834) {};
\node [black,circle,draw,fill=white,scale=0.75,line width=1mm] (43) at (38.26834,92.38795) {};
\node [black,circle,draw,fill=white,scale=0.75,line width=1mm] (44) at (38.26834,-92.38795) {};
\node [black,circle,draw,fill=white,scale=0.75,line width=1mm] (45) at (-92.38795,38.26834) {};
\node [black,circle,draw,fill=white,scale=0.75,line width=1mm] (46) at (92.38795,38.26834) {};
\end{scope}
\end{tikzpicture}

%% file: 3reg_bip_frac_5_4_edges_9_5_7_11.tikz
\begin{tikzpicture}[scale=0.07]
\def\vertexscale{1.80}
\node [circle,blue,draw,scale=\vertexscale,fill=blue] (1) at (100.00000,0.00000) {};
\node [circle,black,draw,scale=\vertexscale,fill=black] (2) at (40.90861,-31.42083) {};
\node [circle,black,draw,scale=\vertexscale,fill=black] (3) at (80.90170,-58.77853) {};
\node [circle,black,draw,scale=\vertexscale,fill=black] (4) at (80.90170,58.77853) {};
\node [circle,blue,draw,scale=\vertexscale,fill=blue] (5) at (60.37616,-44.14427) {};
\node [circle,blue,draw,scale=\vertexscale,fill=blue] (6) at (30.90170,-95.10565) {};
\node [circle,blue,draw,scale=\vertexscale,fill=blue] (7) at (30.90170,95.10565) {};
\node [circle,blue,draw,scale=\vertexscale,fill=blue] (8) at (-80.90170,58.77853) {};
\node [circle,black,draw,scale=\vertexscale,fill=black] (9) at (-30.90170,95.10565) {};
\node [circle,black,draw,scale=\vertexscale,fill=black] (10) at (-30.90170,-95.10565) {};
\node [circle,black,draw,scale=\vertexscale,fill=black] (11) at (-61.84197,-44.98079) {};
\node [circle,black,draw,scale=\vertexscale,fill=black] (12) at (-100.00000,-0.00000) {};
\node [circle,blue,draw,scale=\vertexscale,fill=blue] (13) at (-42.02132,-30.70565) {};
\node [circle,blue,draw,scale=\vertexscale,fill=blue] (14) at (-80.90170,-58.77853) {};
\draw [black] (1) to (2);
\draw [black] (1) to (4);
\draw [black] (1) to (3);
\draw [black] (2) to (5);
\draw [black] (2) to (6);
\draw [black] (3) to (6);
\draw [black] (3) to (5);
\draw [black] (4) to (8);
\draw [black] (4) to (7);
\draw [black] (6) to (10);
\draw [black] (7) to (9);
\draw [black] (8) to (9);
\draw [black] (8) to (12);
\draw [black] (10) to (13);
\draw [black] (10) to (14);
\draw [black] (11) to (14);
\draw [black] (11) to (13);
\draw [black] (12) to (13);
\draw [black] (12) to (14);
\draw [->,green,dashed,line width=2mm] (5) to (9);
\draw [->,red,dashed,line width=2mm] (11) to (7);
\end{tikzpicture}

%% file: 4reg_planar_4con_frac_6_5.tikz
\begin{tikzpicture}[scale=0.07]
\def\vertexscale{0.77}
\node [circle,black,draw,scale=\vertexscale,fill=black] (1) at (5.30199,-7.23807) {};
\node [circle,black,draw,scale=\vertexscale,fill=black] (2) at (14.47774,5.94078) {};
\node [circle,black,draw,scale=\vertexscale,fill=black] (3) at (6.15159,21.32619) {};
\node [circle,black,draw,scale=\vertexscale,fill=black] (4) at (-5.60495,-7.23841) {};
\node [circle,black,draw,scale=\vertexscale,fill=black] (5) at (10.89524,-19.67864) {};
\node [circle,black,draw,scale=\vertexscale,fill=black] (6) at (24.23394,15.58562) {};
\node [circle,black,draw,scale=\vertexscale,fill=black] (7) at (14.52196,40.28584) {};
\node [circle,black,draw,scale=\vertexscale,fill=black] (8) at (-6.15160,21.32644) {};
\node [circle,black,draw,scale=\vertexscale,fill=black] (9) at (-14.47777,5.94081) {};
\node [circle,black,draw,scale=\vertexscale,fill=black] (10) at (-10.89507,-19.67887) {};
\node [circle,black,draw,scale=\vertexscale,fill=black] (11) at (13.77565,-40.19026) {};
\node [circle,black,draw,scale=\vertexscale,fill=black] (12) at (24.56256,-16.20065) {};
\node [circle,black,draw,scale=\vertexscale,fill=black] (13) at (36.54506,-0.03231) {};
\node [circle,black,draw,scale=\vertexscale,fill=black] (14) at (45.14545,39.16472) {};
\node [circle,black,draw,scale=\vertexscale,fill=black] (15) at (-14.52162,40.28577) {};
\node [circle,black,draw,scale=\vertexscale,fill=black] (16) at (-24.23391,15.58555) {};
\node [circle,black,draw,scale=\vertexscale,fill=black] (17) at (-24.56257,-16.20065) {};
\node [circle,black,draw,scale=\vertexscale,fill=black] (18) at (-13.77532,-40.19017) {};
\node [circle,black,draw,scale=\vertexscale,fill=black] (19) at (45.07301,-39.17392) {};
\node [circle,black,draw,scale=\vertexscale,fill=black] (20) at (-45.14501,39.16562) {};
\node [circle,black,draw,scale=\vertexscale,fill=black] (21) at (-36.54505,-0.03232) {};
\node [circle,black,draw,scale=\vertexscale,fill=black] (22) at (-45.07258,-39.17479) {};
\draw [black] (1) to (2);
\draw [black] (1) to (5);
\draw [black] (1) to (4);
\draw [black] (1) to (3);
\draw [black] (2) to (6);
\draw [black] (2) to (5);
\draw [black] (2) to (3);
\draw [black] (3) to (8);
\draw [black] (3) to (7);
\draw [black] (4) to (9);
\draw [black] (4) to (8);
\draw [black] (4) to (10);
\draw [black] (5) to (11);
\draw [black] (5) to (12);
\draw [black] (6) to (7);
\draw [black] (6) to (13);
\draw [black] (6) to (12);
\draw [black] (7) to (14);
\draw [black] (7) to (15);
\draw [black] (8) to (9);
\draw [black] (8) to (15);
\draw [black] (9) to (16);
\draw [black] (9) to (10);
\draw [black] (10) to (18);
\draw [black] (10) to (17);
\draw [black] (11) to (19);
\draw [black] (11) to (18);
\draw [black] (11) to (12);
\draw [black] (12) to (13);
\draw [black] (13) to (14);
\draw [black] (13) to (19);
\draw [black] (14) to (19);
\draw [black] (15) to (16);
\draw [black] (15) to (20);
\draw [black] (16) to (21);
\draw [black] (16) to (17);
\draw [black] (17) to (18);
\draw [black] (17) to (21);
\draw [black] (18) to (22);
\draw [black] (20) to (21);
\draw [black] (20) to (22);
\draw [black] (21) to (22);
\draw[->,green,dashed, line width=1.2mm]   (19) to (70.07301,-39.17392);
\draw[->,red,dashed, line width=1.2mm]   (14) to (70.07301,39.17392);
\draw[->,green,dashed, line width=1.2mm]  (-70.07301,-39.17392) to (22);
\draw[->,red,dashed, line width=1.2mm]  (-70.07301,39.17392) to (20);
\end{tikzpicture}